\documentclass[11pt,leqno,oneside]{amsart}
\usepackage{layout}
\usepackage{mathrsfs,dsfont}
\usepackage{amsmath,amstext,amsthm,amssymb,bbm,color}
\usepackage{comment}
\usepackage{charter}
\usepackage{typearea}
\usepackage{pdfsync}
\usepackage[width=6.5in,height=8.7in]{geometry}

\usepackage{amsfonts}

\usepackage[utf8]{inputenc}

\def\bt{\begin{thm}}
\def\et{\end{thm}}
\def\bl{\begin{lem}}
\def\el{\end{lem}}
\def\bd{\begin{defn}}
\def\ed{\end{defn}}
\def\bc{\begin{cor}}
\def\ec{\end{cor}}
\def\bp{\begin{proof}}
\def\ep{\end{proof}}
\def\br{\begin{rem}}
\def\er{\end{rem}}

\newtheorem{thm}{Theorem}[section]
\newtheorem{prop}[thm]{Proposition}
\newtheorem{lem}[thm]{Lemma}
\newtheorem{defn}[thm]{Definition}

\newtheorem{rem}[thm]{Remark}
\newtheorem{cor}[thm]{Corollary}

\numberwithin{equation}{section}

\newcommand{\bthm}{\begin{thm}}
\newcommand{\ethm}{\end{thm}}
\newcommand{\bstp}{\begin{stp}}
\newcommand{\estp}{\end{stp}}
\newcommand{\blemma}{\begin{lemma}}
\newcommand{\elemma}{\end{lemma}}
\newcommand{\bprop}{\begin{prop}}
\newcommand{\eprop}{\end{prop}}
\newcommand{\bpf}{\begin{pf}}
\newcommand{\epf}{\end{pf}}
\newcommand{\bdefn}{\begin{defn}}
\newcommand{\edefn}{\end{defn}}
\newcommand{\brk}{\begin{rmrk}}
\newcommand{\erk}{\end{rmrk}}
\newcommand{\bcrl}{\begin{crl}}
\newcommand{\ecrl}{\end{crl}}

\usepackage{amsfonts}

\title[]{A central limit theorem associated with a sequence of positive line bundles}

\address{}
\address{Faculty of Engineering and Natural Sciences, Sabanc{\i} University, \.{I}stanbul, Turkey}
\email{afrimbojnik@sabanciuniv.edu}
\email{ozangunyuz@alumni.sabanciuniv.edu}

\date{}


\begin{document}

\author{Afrim Bojnik}
\author{Ozan Günyüz}

\thanks{A.\ Bojnik was partially supported by T\"{U}B\.{I}TAK grant ARDEB-3501/118F049}

\begin{abstract}
We  prove a central limit theorem for smooth linear statistics associated with zero divisors of standard Gaussian holomorphic sections in a sequence of holomorphic line bundles with Hermitian metrics of class $\mathscr{C}^{3}$ over a compact K\"{a}hler manifold. In the course of our analysis, we derive first-order asymptotics and upper decay estimates for near and off-diagonal Bergman kernels, respectively. These results are essential for determining the statistical properties of the zeros of random holomorphic sections.
  \end{abstract}
\maketitle


\section{Introduction}\label{S2}

Over recent years, there has been a growing interest in probabilistic challenges within the scope of complex algebraic or analytic geometry. Such problems initially arose from studying the zeros of random polynomials and other analytic functions, analyzed by various authors under different assumptions on random coefficients. The field gained momentum with the foundational works of Bloch-Pólya\cite{BP}, Littlewood-Offord \cite{LO43}, Hammersley \cite{HAM56}, Kac \cite{Kac43}, and Erdös-Turán \cite{ET50}, and has now become a classical area of study.  A significant breakthrough was achieved by Shiffman and Zelditch \cite{SZ99}, who extended these ideas into complex geometric setting. They demonstrated that zeros of random sections of high tensor powers of a positive line bundle $L$ on a projective manifold $X$ tend to become uniformly distributed in accordance with the natural measure from line bundle $L$. In a series of papers, these authors further delved into the correlations, variance, expected distribution and other statistical properties related to zeros of holomorphic sections, as elaborated in references \cite{BSZ,SZ08,SZ10}. In this framework, Dinh and Sibony \cite{DS06} innovated a method to understand zero distribution by applying formalism of meromorphic transforms from complex dynamics, and set convergence speed bounds in the compact case, enhancing Shiffman and Zelditch's initial results. For a comprehensive overview, especially in the light of more general probability distributions and geometric backgrounds, we invite the reader to look into the survey \cite{BCHM}.
Furthermore, it is important to emphasize that random polynomials and holomorphic sections have a significant role in theoretical physics. Holomorphic random sections, in particular, are employed as an important model for quantum chaos, leading to extensive research on the distribution of their zero points by physicists, see for example \cite{BBL, Hann, NV98}.

Alongside these developments, investigation of the central limit theorem in the context of smooth linear statistics, such as integrals of smooth test forms over zero divisors of random holomorphic sections, is another intriguing challenge. In this area of research, the work by Sodin and Tsirelson \cite{ST} is particularly influential. They established the following asymptotic normality result for Gaussian random polynomials and analytic functions in the complex plane. This result stands as a vital instrument in our study as well.

\subsection{Theorem of Sodin and Tsirelson}
Given a sequence $\{\gamma_{j}\}_{j=1}^{\infty}$ of complex-valued measurable functions on the measure space ($G, \sigma$) such that \begin{equation}\sum_{j=1}^{\infty}{|\gamma_j(x)|^{2}}=1\,\,\text{for}\,\,\text{any}\,\,x\in G, \end{equation} following \cite{ST} (and also \cite{SZ10}), a \textit{normalized complex Gaussian process} is defined to be a complex-valued random function $\alpha(x)$ on a measure space $(G, \sigma)$ in the following form \begin{equation}\label{gpro}
                                                      \alpha(x)=\sum_{j=1}^{\infty}{b_{j} \gamma_j(x)},
                                                    \end{equation}where the coefficients $b_{j}$ are i.i.d. centered complex Gaussian random variables with variance one.
                                                 The covariance function of $\alpha(x)$ is defined by \begin{equation}\label{cov}
                                                 	\mathcal{C}(x, y)= \mathbb{E}[\alpha(x) \overline{\alpha(y)}]= \sum_{j=1}^{\infty}{\gamma_{j}(x) \overline{\gamma_{j}(y)}}. \end{equation}
A simple observation gives that $|\mathcal{C}(x, y)| \leq 1$ and $\mathcal{C}(x, x)=1$.

Consider a sequence $\{\alpha_{j}\}_{j=1}^{\infty}$ of normalized complex Gaussian processes on a finite measure space $(G, \sigma)$, and let $\lambda(\rho)\in L^2(\mathbb{R}^{+}, e^{\frac{-\rho^2}{2}}\rho d\rho)$. Suppose $\phi: G \rightarrow \mathbb{R}$ is a bounded and measurable function, we will focus on the following non-linear functionals that also serve as random variables in this context. \begin{equation} \label{nonlin} \mathcal{F}^{\phi}_{p}(\alpha_{p})=\int_{G}{\lambda(|\alpha_{p}(x)|)\phi(x)d\sigma(x)}. \end{equation}

The next theorem (Theorem 2.2 of \cite{ST}) was proved by Sodin and Tsirelson.

\begin{thm}\label{sots}
For each $p=1, 2, \ldots$, let $\mathcal{C}_{p}(r, s)$ be the covariance functions for the complex Gaussian processes. Assume that the two conditions below hold for all $\nu \in \mathbb{N}$:

\begin{itemize}
 \item[(i)] $$\liminf_{p\rightarrow \infty}{\frac{\int_{G}\int_{G}{|\mathcal{C}_{p}(r, s)|^{2\nu}\phi(r)\phi(s)d\sigma(r)d\sigma(s)}}{\sup_{r\in G} \int_{G}{|\mathcal{C}_{p}(r, s)|d\sigma(s)}}}>0.$$
  \item[(ii)] $$\lim_{p\rightarrow \infty}{\sup_{r\in G} \int_{G}{|\mathcal{C}_{p}(r, s)|d\sigma(s)}}=0.$$
\end{itemize}

Then the distributions of the random variables \begin{equation*}
                                                 \frac{\mathcal{F}^{\phi}_{p}(\alpha_{p})- \mathbb{E}[\mathcal{F}^{\phi}_{p}(\alpha_{p})]}{\sqrt{\mathrm{Var}[\mathcal{F}^{\phi}_{p}(\alpha_{p})]}}
                                               \end{equation*} converges weakly to the normal distribution $\mathcal{N}(0, 1)$ as $p \rightarrow \infty$. If $\lambda$ is increasing, then it is sufficient for $(i)$ to hold only for $\nu=1$.
\end{thm}

The proof is based on applying the method of moments, a fundamental tool in probability theory, coupled with the use of the diagram technique. This approach facilitates the computation of moments for non-linear functionals, which are then compared to the moments found in the standard Gaussian distribution. Such an approach is a classical one in establishing the central limit theorem for non-linear functionals within Gaussian fields. We also remark that the condition (ii) ensures that $\mathrm{Var}[\mathcal{F}_{p}^{\phi}(\alpha_{p})]\rightarrow 0$ as $p\rightarrow \infty$.

There are two extensions of Theorem \ref{sots} in different contexts. The first, due to the work of Shiffman and Zelditch (\cite{SZ10}), applies to the prequantum line bundle setting. This involves random holomorphic sections of a Hermitian line bundle (with $\mathscr{C}^{3}$ Hermitian metrics) over a compact Kähler manifold, where the first Chern form and K\"{a}hler form satisfy the prequantum line bundle condition. The second extension is studied in \cite{Bay17}, where Bayraktar expanded Theorem \ref{sots} to encompass general random polynomials in $\mathbb{C}^{n}$ with techniques of weighted pluripotential theory.

 The authors of the current paper have recently proved a central limit theorem for smooth linear statistics related to the zero divisors of random holomorphic sections of large tensor powers of a positive Hermitian holomorphic line bundle over a non-compact Hermitian manifold, see \cite{BG1} for further details.

In all these adaptations, the asymptotic behavior of the normalized Bergman kernel, particularly off-diagonal and near-diagonal, proves to be critical. This kernel, functioning as a covariance function of normalized complex Gaussian processes, is pivotal to the analyses. Separately, Nazarov and Sodin (\cite{NS}) explored the asymptotic normality of linear statistics of zeros in Gaussian entire functions on $\mathbb{C}$. Their approach is broader, considering measurable bounded test functions and the clustering of $k$-point correlation functions.

The primary objective of the current study is to apply Theorem \ref{sots} to standard Gaussian holomorphic sections in a sequence of positive holomorphic line bundles with Hermitian metrics of class $\mathscr{C}^{3}$ on a compact Kähler manifold. The following is the main result of the paper.

\begin{thm}\label{maint}
Let $\{(L_{p}, h_{p})\}^{\infty}_{p\geq 1}$ be a sequence of positive holomorphic line bundles over a compact K\"{a}hler manifold $(X, \omega)$ of dimension $n$ with diophantine condition (\ref{appr}) and Hermitian metrics of class $\mathscr{C}^{3}$ such that $\frac{\|h_{p}\|^{1/3}_{3}}{\sqrt{A_{p}}} \rightarrow 0$ as $p\rightarrow \infty$. Suppose that $H^{0}(X, L_{p})$ is endowed with the standard Gaussian probability measure for all $p \geq 1$. For $s_{p}\in H^{0}(X, L_{p})$ and a real valued $(n-1, n-1)$-form  $\phi$ on $X$ with $\mathscr{C}^{3}$-coefficients and $dd^{c} \phi \not\equiv 0$, the distributions of the random variables \begin{equation}\label{clte}
                                                                       \frac{\langle [Z_{s_{p}}], \phi \rangle - \mathbb{E}\langle [Z_{s_{p}}], \phi \rangle}{\sqrt{\mathrm{Var}\langle [Z_{s_{p}}], \phi \rangle}}
                                                                     \end{equation} weakly converge towards the standard (real) Gaussian distribution $\mathcal{N}(0, 1)$ as $p\rightarrow \infty$.

\end{thm}

As seen from Theorem \ref{maint}, the non-linear functionals $\mathcal{F}_{p}^{\phi}$ in Theorem \ref{sots} are replaced by so-called \textit{smooth linear statistics} (see Section \ref{lastsec}) in the current setting.

To achieve this, we start by setting up our geometric and probabilistic framework, building upon the foundations in earlier studies \cite{CLMM}, \cite{CMM}. The key to our analysis is the upper decaying estimate of the off-diagonal Bergman kernel and the first order asymptotics of the Bergman kernel function. In this context, for the asymptotic distribution and variance estimates of smooth linear statistics of random zero sets in any codimension under more general probability measures, we refer the interested reader to \cite{BG}.

\section{Background} \label{Sec2}

Let $(X, J, \omega)$ be a compact K\"{a}hler manifold of $\dim_{\mathbb{C}}{X}=n$, with the complex structure $J$, and the K\"{a}hler form $\omega$ on it. Let $\{(L_{p},h_{p})\}^{\infty}_{p \geq 1}$ be a sequence of positive line bundles with Hermitian $\mathscr{C}^{3}$-metrics $h_{p}$. In this paper, we employ the concepts of Nakano and Griffiths positivity for vector bundles. It is important to note that Nakano positivity implies Griffiths positivity in vector bundles, though the reverse is not always true (refer to \cite{Dem12} for more details). However, in the case of line bundles, these two notions of positivity coincide.  Therefore, we define positivity for a holomorphic line bundle (Hermitian, but not necessarily with $\mathscr{C}^{\infty}$-metrics), denoted as $(L, h)$, in the following way: a line bundle is considered positive or semi-positive if its local weight functions
 $\varphi$, which are of class $\mathscr{C}^{2}$ and corresponding to the metric $h$, are strictly plurisubharmonic and plurisubharmonic, respectively.

In the present setting, the associated local weight functions $\varphi$ are already of class $\mathscr{C}^{3}$ and by Nakano-Griffiths positivity, they will also satisfy $dd^{c} \varphi >0$ at every point. This also indicates, by Grauert's result (see, for instance, Proposition 2.4 of [CMM]), that such line bundles are ample and $X$ is projective

We denote the global holomorphic sections of $L_p$ by $H^{0}(X, L_{p})$. In this context, we consider an inner product on the space of smooth sections $\mathscr{C}^{\infty}(X,L_{p})$, using the metric $h_{p}$ and the volume form $\frac{\omega^{n}}{n!}$ on $X$. This inner product is defined  as \begin{equation} \label{innp}\langle s_{1}, s_{2} \rangle_{p}:=\int_{X}{\langle s_{1}(x), s_{2}(x)\rangle_{h_{p}} \frac{\omega^{n}}{n!}},\end{equation} and the norm of a section $s$ is given by $\|s\|^{2}_{p}:=\langle s, s \rangle_{p}$.
Based on the Cartan-Serre finiteness theorem, we know that $H^{0}(X, L_{p})$ is finite-dimensional. We denote the dimension of this space as $d_{p}:= \dim{H^{0}(X, L_{p})}$. Subsequently, we consider $\mathcal{L}^{2}(X, L_{p})$, which is the completion of $\mathscr{C}^{\infty}(X,L_{p})$ under this norm, forming a Hilbert space of square-integrable sections of $L_{p}$.

Next, we define the orthogonal projection operator $K_{p}:\mathcal{L}^{2}(X, L_{p}) \rightarrow H^{0}(X, L_{p})$. The Bergman kernel, $K_{p}(x, y)$, is the integral kernel of this projection. If $\{S^{p}_{j}\}_{j=1}^{d_{p}}$ is an orthonormal basis for $H^{0}(X, L_{p})$, by using reproducing property of $K_{p}(x, y)$, we express $K_{p}(x,y)$ in terms of this basis as follows:  \begin{equation}\label{offds} K_{p}(x, y)=\sum_{j=1}^{d_{p}}{S^{p}_{j}(x)\otimes S^{p}_{j}(y)^{*}}\in L_{p, x}\otimes L_{p, y}^{*},\end{equation}where $S^{p}_{j}(y)^{*}=\langle \, .\,, S^{p}_{j}(y)\rangle_{h_{p}}\in L^{*}_{p, y}$. The restriction of the Bergman kernel to the diagonal of $X$ is called the Bergman kernel function of $H^{0}(X, L_{p})$, which we denote by $K_{p}(x):= K_{p}(x, x)$, and (\ref{offds}) becomes  \begin{equation}\label{diags}K_{p}(x)=\sum_{j=1}^{d_{p}}{|S^{p}_{j}(x)|_{h_{p}}^{2}}.\end{equation}The Bergman kernel function has the dimensional density property,  namely $$\int_{X}K_{p}(x)\frac{\omega^{n}}{n!}=d_{p}.$$
In addition, it satisfies the following variational principle
\begin{equation}\label{variat}
	K_{p}(x) = \max \{ |S(x)|^2_{h_{p}} : S \in H^0(X, L_{p}), \ ||S||_{p} = 1 \}.\end{equation}
This holds for every $x \in X$ for which $\varphi_{p}(x) > -\infty$, with $\varphi_{p}$ denoting a local weight, defined above, for the metric $h_{p}$ in the vicinity of $x$.

We also define the normalized Bergman kernel as follows: \begin{equation*}
	\widehat{\mathcal{K}}_{p}(x, y):= \frac{|K_{p}(x, y)|_{h_{p, x} \otimes h^{*}_{p, y}}}{\sqrt{K_{p}(x)} \sqrt{K_{p}(y)}}.
\end{equation*}
which will be important throughout.

\subsection{Reference Covers} \label{metricss}

 $(U, z),\,\,z=(z_{1}, \ldots, z_{n}),$ will indicate local coordinates centered at a point $x\in X$. The closed polydisk around $y\in U$ of equilateral radius $(r, \ldots, r), \,r>0,$ is given by \begin{equation*}
                                                 P^{n}(y, r):=\{z\in U: |z_{j}- y_{j}| \leq r, \,\,j=1, 2, \ldots, m\}.
                                               \end{equation*} The coordinates $(U, z)$ are said to be K\"{a}hler at $y\in U$ in case \begin{equation} \label{kahcor}
                                                                                      \omega_{z}=\frac{i}{2}\sum_{j=1}^{m}{dz_{j} \wedge d\overline{z_{j}}} + \frac{i}{2}\sum_{j, k}{O(|z-y|^{2})dz_{j} \wedge d\overline{z_{k}}} \,\,\,\text{on}\,\,\,U.
                                                                                    \end{equation}

\begin{defn}
  \label{refed} A \textit{reference cover} of $X$ is defined as follows: for $j=1, 2, \ldots, N$, a set of points $x_{j}\in X$ and
\begin{itemize}
  \item[(a)] Stein open simply connected coordinate neighborhoods $(U_{j}, w^{(j)})$ centered at $x_{j}\equiv 0$.
  \item[(b)] $R_{j}>0$ such that $P^{n}(x_{j}, 2R_{j}) \Subset U_{j}$ and for every $y\in P^{n}(x_{j}, 2R_{j})$ there exist coordinates on $U_{j}$ which are K\"{a}hler at $y$.
  \item[(c)]$X=\bigcup_{j=1}^{N}{P^{n}(x_{j}, R_{j})}$.
\end{itemize}We will write $R=\min{R_{j}}$ once a reference cover is provided. \end{defn}It is not difficult to see how one can construct a reference cover. Indeed, first, for $x\in X$, take a Stein open simply connected neighborhood (for instance, a round ball in $\mathbb{C}^{n}$) $U$ of $0 \in \mathbb{C}^{n}$, where $x \equiv 0$ under a determined chart. Choose some $R >0$ so that $P^{n}(x, R) \Subset U$ and for every $y\in P^{n}(x, R)$ there exist K\"{a}hler  coordinates $(U, z)$ at $y$. The compactness of $X$ implies that there exist finitely many points $\{x_{j}\}_{j=1}^{N}$ such that the three conditions above are satisfied.

We take into consideration the differential operators $D^{\alpha}_{w},\,\,\alpha \in \mathbb{N}^{2n}$ on $U_{j}$, corresponding to the real coordinates associated to $w=w^{j}$.  For $\varphi \in \mathscr{C}^{k}(U_{j})$, we define \begin{equation} \label{albeolm}
                                                                   \|\varphi\|_{k}=\|\varphi\|_{k, w}=\sup{\{|D^{\alpha}_{w} \varphi(w)|: w\in P^{n}(x_{j}, 2R_{j}), \,|\alpha|\leq  k\}}.
                                                                 \end{equation}
Let $(L, h)$ be a Hermitian holomorphic line bundle on $X$, i.e., the metric $h$ is smooth. For $k \leq l$, write \begin{equation*}
                  \|h\|_{k, U_{j}}= \inf{\{\|\varphi_{j}\|_{k}: \varphi_{j} \in \mathscr{C}^{l}(U_{j}) \,\,\text{is}\,\,\text{a}\,\,\text{weight}\,\,\text{of}\,\,h\,\,\text{on}\,\,U_{j}\}},
                \end{equation*} and
\begin{equation*}
\|h\|_{k}=\max{\{1, \|h\|_{k, U_{j}}: 1\leq j \leq N\}}.
\end{equation*}

 $\varphi_{j}$ is said to be a weight of $h$ on $U_{j}$ if there exists a holomorphic frame $e_{j}$ of $L$ on $U_{j}$ such that $|e_{j}|_{h}=e^{-\varphi_{j}}$.

\begin{lem} \label{refe}
	Let a reference cover of $X$ be given. Then there exists a constant $D>0$ relying on the reference cover with the following property: When provided with any Hermitian line bundle $(L, h)$ on $X$, any $j\in \{1, \ldots, N\}$ and any $x\in P^{n}(x_{j}, R_{j})$, there exist coordinates $z=(z_{1}, \ldots, z_{n})$ on $P^{n}(x, R)$ which are centered at $x \equiv 0$ and K\"{a}hler coordinates for $x$ such that \begin{itemize}
		\item[(i)]$ dV \leq (1+Dr^{2})\frac{\omega^{n}}{n!}$ and $\frac{\omega^{n}}{n!} \leq (1+ Dr^2) dV$ hold on $P^{n}(x, r)$ for any $r<R$, where $dV=dV(z)$ is the Euclidean volume relative to the coordinates $z$,
		\item[(ii)] $(L, h)$ has a weight $\varphi$ on $P^{n}(x, R)$ with $\varphi(z)=\Re{t(z)}+\sum_{j=1}^{n}{\lambda_{j}|z_{j}|^{2}}+ \widetilde{\varphi}(z)$, where $t$ is a holomorphic polynomial of degree at most $2$, $\lambda_{j}\in \mathbb{R}$ and $|\widetilde{\varphi}(z)| \leq D' \|h\|_{3} |z|^{3}$ for $z\in P^{n}(x, R)$.
\end{itemize} \end{lem}

\begin{proof}
	By the definition of a reference cover, there exist coordinates $z$ on $U_{j}$ which are Kähler for $x \in P^{n}(x_{j}, R_{j})$. Then $\omega_{z} = \frac{i}{2}\sum_{l=1}^n dz_l \wedge d\overline{z}_l  + \frac{i}{2}\sum_{j, k}{O(|z-x|^{2})\,dz_{j} \wedge d\overline{z_{k}}} $ and (i) holds with a constant $D_{j}$ uniform for $x \in P^{n}(x_{j}, R_{j})$. Let $e_{j}$ be a frame of $L$ on $U_{j}$, $\varphi$ a weight of $h$ on $U_{j}$ with $|e_{j}|_h = e^{-\varphi}$ and $||\varphi||_{3,z} \leq 2||h||_{3}$. By translation, we may assume $x = 0$ and write $\varphi(z) = \Re{t(z)} + \varphi_{2}(z) + \varphi_{3}(z)$, where $t(z)$ is a holomorphic polynomial of degree $\leq 2$ in $z$, $\varphi_{2}(z) = \sum_{k,l=1}^n \mu_{kl}z_k\overline{z}_l$ and $\Re{f(z)} + \varphi_{2}(z)$ is the Taylor polynomial of $\varphi$ of order $2$ at $0$. In order to estimate $\varphi_{3}(z)$, let $\|\varphi\|_{3, z}$ be the supremum norm of the derivatives of $\varphi$ of order $3$ on $P^{n}(x_{j}, R_{j})$ in the $z$-coordinates. Then, by (\ref{albeolm}), there exists a constant $D'_{j}$ being uniform on $P^{n}(x_{j}, R_{j})$ such that  $||\varphi||_{3,z} \leq D'_j||\varphi||_{3,w}\leq 2D'_j||h||_{3}$, which also gives that $|\varphi_{3}(z)| \leq  2 D'_{j} \|h\|_{3} |z|^{3}$ for all $z\in P^{n}(x, R)$.
	
	Applying a unitary change of coordinates, we may suppose that $\varphi(\zeta)=\Re{t(\zeta)} + \sum_{j=1}^{n}{\lambda^{p}_{j}|\zeta_{j}|^{2}}+ \widetilde{\varphi}(\zeta)$. Under these coordinates, $\frac{w^{n}}{n!}$ and $\widetilde{\varphi}(\zeta)$ verify the required estimates with a uniform constant $D_{j}$ for $x \in P^{n}(x_{j}, R_{j})$, as unitary transformations preserve distances. Finally putting $D' = \max_{1 \leq j \leq N} D'_{j}$ finishes the proof.
	
\end{proof}
Let $(X, \omega)$ be a compact K\"{a}hler manifold of $\dim_{\mathbb{C}}{X}=n$. Following \cite{CLMM}, the diophantine approximation of the K\"{a}hler form $\omega$ is defined as follows:
\begin{equation}\label{appr}\frac{1}{A_{p}}c_{1}(L_{p}, h_{p}) = \omega + O(A_{p}^{-a}) \,\,\text{in}\,\,\text{the}\,\,\mathscr{C}^{0}\text{-topology}\,\,\text{as}\,\,p\rightarrow \infty, \end{equation}where $a>0$, $A_{p}>0$ and $\lim_{p\rightarrow \infty}{A_{p}}= \infty$. This means that \begin{equation}\label{cur}
\left \| \frac{1}{A_{p}}c_{1}(L_{p}, h_{p})-\, \omega \right \|_{\mathscr{C}^{0}} = O(A^{-a}_{p}),
 \end{equation}where $A_{p}> 0$,\, $a>0$ and  $\lim_{p\rightarrow \infty}{A_{p}}=+\infty$. This was first considered in \cite{CLMM} in the $\mathscr{C}^{\infty}$-norm topology induced by the Levi-Civita connection $\nabla^{TX}$ because the authors deal with the complete asymptotic expansion of the Bergman kernel restricted to the diagonal. We do not need such a strong topology, in fact, only the $\mathscr{C}^{0}$-norm (or continuous norm) topology will be sufficient for our aims.

 Let $\{U_{j}\}^{N}_{j=1}$ be a finite subcover of $X$. Locally, on each $U_{j}$, we have the following representations \begin{equation}\label{locch}
                                               \frac{1}{A_{p}}c_{1}(L_{p}, h_{p})(z)=i \sum_{k, j}{\frac{1}{ \pi} \frac{1}{A_{p}} \alpha_{kj}(z) dz_{k} \wedge d\overline{z_{j}}}
                                              \end{equation} and \begin{equation}\label{lock}
                                                                   \omega(z)=i \sum_{k, j}{\chi_{kj}(z) dz_{k} \wedge d\overline{z_{j}}}.
                                                                \end{equation}Here $[\chi_{kj}(z)]_{kj}$ is a positive definite Hermitian matrix, because in a K\"{a}hler manifold, we always have a strictly plurisubharmonic local potential function $\psi$ so that $$\chi_{kj}(z)= \frac{\partial^{2} \psi}{\partial z_{j} \partial \overline{z_{k}}}(z) \, \, \text{for each} \,  1 \leq k,\,j\,\leq n.$$ Similarly, since line bundles $L_{p}$ are positive, by the definition of positivity, $[\alpha_{kj}(z)]_{kj}=\big[\frac{\partial^{2} \varphi_{p}}{\partial z_{j} \partial \overline{z_{k}}}(z)\big]_{kj}$  is a positive definite Hermitian matrix , where $\varphi_{p}$ is the corresponding local weight function for $h_{p}$.

 By using the diophantine approximation (\ref{cur}) and the positivity of the above two forms, one can find some large enough $p_{0} \in \mathbb{N}$ such that

    \begin{equation}\label{diops}
    \frac{3 \omega}{4} \leq \frac{1}{A_{p}}c_{1}(L_{p}, h_{p}) \leq \frac{5\omega}{4}
    \end{equation} for $p \geq p_{0}$. This will be useful in the proofs of Theorem \ref{bkad} and Theorem \ref{udeb}.

We also observe that, at the point $x \equiv 0$ where we have the K\"{a}hler coordinates by (\ref{kahcor}),  we have  \begin{equation*}\label{kahloc}
                                                                                      \omega_{x}= i \sum_{j=1}^{n}{\frac{1}{2} dz_{j} \wedge d\overline{z_{j}}}.
                                                                                    \end{equation*}Also, by using the local representation of $c_{1}(L_{p}, h_{p})$ and Lemma \ref{refe},  \begin{equation}\label{locchern}c_{1}(L_{p}, h_{p})_{x}=dd^{c} \varphi_{p}(0)=i \sum_{j=1}^{n}{\frac{\lambda^{p}_{j}}{\pi} dz_{j} \wedge d\overline{z_{j}}}.\end{equation}

Diophantine approximation (\ref{appr}) implies \begin{equation} \label{limitl}
\lim_{p\rightarrow \infty}{\frac{\lambda^{p}_{j}}{\pi A_{p}}}= \frac{1}{2}\,\,\,\text{for}\,\,j=1, 2, \ldots, n,
\end{equation}which in turn gives \begin{equation}\label{limitl1}
                            \lim_{p\rightarrow \infty}{\frac{\lambda^{p}_{1} \ldots \lambda^{p}_{n}}{A_{p}^{n}}}= (\frac{\pi}{2})^{n}.
                          \end{equation}

In order to measure the distance between any two points $x, y$ on the compact K\"{a}hler manifold $(X, \omega)$, we use the Riemannian distance, which is defined as follows: As is well-known, the K\"{a}hler form $\omega$ and the complex structure $J$ on $X$ compatible with $\omega$ determine a Riemannian metric $g$ on $X$ by $g(u, v):= \omega(u, Jv)$ for all $u, v \in TX$. Given a piecewise smooth curve $\gamma: [a, b] \rightarrow X$ with $\gamma(a)=x$ and $\gamma(b)=y$, the length $L(\gamma)$ of the curve $\gamma$ is given by \begin{equation*}\label{leng}
                                                                                            L(\gamma)=\int_{a}^{b}{\sqrt{g_{\gamma(t)}(\dot{\gamma}(t), \dot{\gamma}(t))} dt}
                                                                                          \end{equation*} and the Riemannian distance $d$ is defined by
\begin{equation*}
d(x, y)=\inf{\{L(\gamma): \gamma(a)=x,\,\,\gamma(b)=y\}}.
\end{equation*}

\section{Demailly's $L^{2}$-estimations for $\overline{\partial}$ operator}
In order to prove both the upper decay estimate of the Bergman kernel and the first-order asymptotics of the Bergman kernel function in our current diophantine setting, we follow the approaches in \cite{CM}, \cite{CMM} and \cite{BCM} to provide first certain $L^{2}$ estimations for solutions of the $\overline{\partial}$-equation, and then derive a weighted estimate for these solutions.

\begin{thm}(\cite{Dem82}, Th\'{e}or\`{e}me 5.1)\label{Demailly82} Let $(X, \omega)$ be a K\"{a}hler manifold with $\dim_{\mathbb{C}}{X}=n$ having a complete K\"{a}hler metric. Let $(L, h)$ be a singular Hermitian holomorphic line bundles such that $c_{1}(L, h) \geq 0$. Then for any form $g\in L^2_{n, 1}(X, L, loc)$ verifying \begin{equation}\label{esti1}
                                                                                                                          \overline{\partial}g=0,\,\,\,\int_{X}{|g|^{2}_{h} \frac{\omega^{n}}{n!}} < \infty,
                                                                                                                        \end{equation}there is $u\in L^{2}_{n,0}(X, L, loc)$ with $\overline{\partial}u=g$ such that \begin{equation}\label{esti2}
                                                                                                                          \int_{X}{|u|^{2}_{h} \frac{\omega^{n}}{n!}} \leq  \int_{X}{|g|^{2}_{h} \frac{\omega^{n}}{n!}}.
                                                                                                                        \end{equation}\end{thm}

\begin{thm}\label{De3}
Let $X$ be a complete K\"{a}hler manifold with $\dim_{\mathbb{C}}{X}=n$ and let $\omega$ be a K\"{a}hler form (not necessarily complete) on $X$ such that its Ricci form $\text{Ric}_{\omega} \geq -2\pi T_{0} \omega$ on $X$ for some constant  $T_{0} >0$. Let $(L_{p}, h_{p})$ be a sequence of holomorphic line bundles on $X$ with Hermitian metrics $h_{p}$ of class $\mathscr{C}^{3}$ such that (\ref{appr}) holds and there is a $p_{0}\in \mathbb{N}$ such that $A_{p} \geq 4T_{0}$ for all $p>p_{0}$. If $p>p_{0}$ and $f\in L^{2}_{0, 1}(X, L_{p}, loc)$ satisfies $\overline{\partial}f=0$ and $\int_{X}{|f|^{2}_{h_{p}}\,\frac{\omega^{n}}{n!}}< \infty$, then there exists $u\in L^{2}_{0, 0}(X, L_{p}, loc)$ such that $\overline{\partial}u=f$ and $\int_{X}{|u|^{2}_{h_{p}}\,\frac{\omega^{n}}{n!}}\leq \frac{2}{A_{p}}\int_{X}{|f|^{2}_{h_{p}}\,\frac{\omega^{n}}{n!}}$.
\end{thm}

\begin{proof}
 By the diophantine approximation relation (\ref{appr}), fix some $p_{0} \in \mathbb{N}$ so that the assertions in the theorem are satisfied and also for all $p>p_{0}$, $\frac{3 \omega}{2} \geq \frac{1}{A_{p}}c_{1}(L_{p}, h_{p}) \geq \frac{3 \omega}{4}$. Let $L_{p}=F_{p}\otimes K_{X}$, where $F_{p}=L_{p} \otimes K^{-1}_{X}$. The canonical line bundle $K_{X}$ is endowed with the metric $h^{K_{X}}$ induced by $\omega$. If $g_{p}=h_{p} \otimes h^{K^{-1}_{X}}$ is the induced metric on $F_{p}$, then, since $c_{1}(K_{X}, h_{K_{X}})=-\frac{1}{2 \pi}\text{Ric}_{\omega}$ and $A_{p} \geq 4T_{0}$ for all $p>p_{0}$, \begin{equation}\label{12}
                                                                                             \frac{1}{A_{p}}c_{1}(F_{p}, g_{p})=\frac{1}{A_{p}}c_{1}(L_{p}, h_{p})- \frac{1}{A_{p}}c_{1}(K_{X}, h^{K_{X}})=\frac{1}{A_{p}}c_{1}(L_{p}, h_{p})+ \frac{1}{2\pi A_{p}} Ric_{\omega} \geq \frac{3\omega}{4}-\frac{\omega}{4}=\frac{\omega}{2} \geq 0
                                                                                           \end{equation} for all $p>p_{0}$. On the other hand, there exists a natural isometry, $$\Psi= \sim: \Lambda^{0, q}(T^{*}(X)) \otimes L_{p} \rightarrow \Lambda^{n, q}(T^{*}(X)) \otimes F_{p}$$ by  \begin{equation}\Psi(s)= \tilde{s}=(w^{1}\wedge \ldots \wedge w^{n} \wedge s) \otimes (w_{1} \wedge \ldots \wedge w_{n}),\end{equation} where ${w_{1}, \ldots w_{n}}$ is a local orthonormal frame of $T^{(1, 0)}(X)$ and $\{w^{1}, \ldots, w^{n}\}$ is the dual frame. This operator $\Psi$ commutes with the action of $\overline{\partial}$. Now for a form $f \in L^{2}_{0, 1}(X, L_{p}, loc)$ satisfying $\overline{\partial}f=0$ and $\int_{X}{|f|^{2}_{h_{p}} \frac{\omega^{n}}{n!}} < \infty$, consider $f'=\frac{\sqrt{2}}{\sqrt{A_{p}}} f$. Obviously we have $\overline{\partial} f'=0$ and $\int_{X}{|f'|^{2}_{h_{p}} \frac{\omega^{n}}{n!}} < \infty$. By using the isometry $\Psi$, we can find $\Psi(f')=F' \in L^{2}_{n, 1}(X, F_{p}, loc)$ with $\overline{\partial}F'=\overline{\partial}\Psi(f')=\Psi \overline{\partial}f'=0$ and $\int_{X}{|F'|^{2}_{g_{p}} \frac{\omega^{n}}{n!}} < \infty$ since isometries preserve the $L^{2}$-norm. By Theorem \ref{Demailly82}, there exists $\tilde{f} \in L^{2}_{n, 0}(X, F_{p}, loc)$ such that $\overline{\partial}\tilde{f}=F'$ and $\int_{X}{|\tilde{f}|^{2}_{g_{p}}\frac{\omega^{n}}{n!}} \leq \int_{X}{|F'|^{2}_{g_{p}}\frac{\omega^{n}}{n!}}$. Taking $u:=\Psi^{-1}\tilde{f}$ and $f'=\Psi^{-1}(F')$ finishes the proof since $\Psi^{-1}$ is an isometry as well.
\end{proof}

\begin{thm}\label{De4}
Let $(X, \omega)$ be a compact K\"{a}hler manifold, $\dim_{\mathbb{C}}{X}=n$ and let $\{(L_{p}, h_{p})\}_{p\geq 1}$ be a sequence of holomorphic line bundles on $X$ with $\mathscr{C}^{3}$ Hermitian metrics as before such that the diophantine approximation condition (\ref{appr}) holds. Then there exists $p_{0}\in \mathbb{N}$ such that if $u_{p}$ are real-valued functions of class $\mathscr{C}^{2}$ on $X$ such that \begin{equation}\label{L22}
                                     \|\overline{\partial}u_{p}\|_{L^{\infty}(X)}\leq \frac{\sqrt{A_{p}}}{8},\,\,dd^{c} u_{p} \geq -\frac{A_{p}}{4} \omega,
                                   \end{equation}then \begin{equation}\label{L33}
                                                        \int_{X}{|v|^{2}_{h_{p}}e^{2u_{p}}\frac{\omega^{n}}{n!}} \leq \frac{16}{3A_{p}}\int_{X}{|\overline{\partial}v|^{2}_{h_{p}}e^{2u_{p}}\frac{\omega^{n}}{n!}}
                                                      \end{equation} holds for $p>p_{0}$ and for every $\mathscr{C}^{1}$-smooth section $v$ of $L_{p}$ which is orthogonal to $H^{0}(X, L_{p})$ with respect to the inner product induced by $h_{p}$ and $\omega^{n}$.
\end{thm}

\begin{proof}
As in the proof of Theorem \ref{De3}, via the diophantine convergence assumption (\ref{appr}), we first fix some $p_{0}\in \mathbb{N}$ so that for any (fixed) $p>p_{0}$, one can get $\frac{3\omega}{4} \leq \frac{1}{A_{p}}c_{1}(L_{p}, h_{p}) \leq \frac{3\omega}{2}$ and $A_{p} \geq 4T_{0}$.  The main idea is to use Theorem \ref{De3}. To this end, let us fix a constant $T_{0}>0$ so that $\text{Ric}_{\omega} \geq -2\pi T_{0} \omega$ on $X$. Using the real-valued functions $u_{p}$ given in the assumptions of theorem, we consider the metrics $g_{p}:= e^{-2u_{p}}h_{p}$ on $L_{p}$. From (2.3.5) in \cite{MM1}(p. 98) and the second relation in (\ref{L22}) yield the following
\begin{equation*}
 c_{1}(L_{p}, g_{p})= c_{1}(L_{p}, h_{p})+ dd^{c} u_{p} \geq \frac{3 A_{p} \omega}{4}- \frac{A_{p} \omega}{4}=\frac{A_{p} \omega}{2}.
                                                                                                                           \end{equation*}
If we define an inner product by using $g_{p}$ in $L^{2}(X, L_{p})$ as $( s_{1}, s_{2} )_{g_{p}}:= \int_{X}{\langle s_{1}, s_{2}\rangle_{g_{p}} \frac{\omega^{n}}{n!}}$, we see, by the relation $g_{p}= e^{-2u_{p}}h_{p}$, for every $s\in H^{0}(X, L_{p})$,   \begin{equation*}
               ( e^{u_{p}}v, s )_{g_{p}}=\int_{X}{\langle e^{2u_{p}}v, s \rangle_{g_{p}}\frac{\omega^{n}}{n!}}=\int_{X}{\langle v, s \rangle_{h_{p}} \frac{\omega^{n}}{n!}}=0
             \end{equation*} for every $\mathscr{C}^{1}$-smooth section $v$ of $L_{p}$.

Write \begin{equation} \label{beta}\beta= \overline{\partial}(e^{2u_{p}}v)=e^{2u_{p}}(2 \overline{\partial}u_{p} \wedge v + \overline{\partial}v). \end{equation}

Clearly $\overline{\partial} \beta=0$. By assumptions on $u_{p}$ and $v$, it follows immediately that $\beta \in  L^{2}_{0, 1}(X, L_{p}, loc)$, so by Theorem \ref{De3}, there exists $\tilde{v}\in L^{2}_{0, 0}(X, L_{p}, loc)$ such that $\overline{\partial} \tilde{v}= \beta$ and \begin{equation}\label{ineq}
                                                                                   \int_{X}{|\tilde{v}|^{2}_{g_{p}} \frac{\omega^{n}}{n!}} \leq \frac{2}{A_{p}}\int_{X}{|\beta|^{2}_{g_{p}}\frac{\omega^{n}}{n!}}.
                                                                                 \end{equation}Since $e^{2u_{p}}v$ is orthogonal to $H^{0}(X, L_{p})$ for every $v\in \mathscr{C}^{1}(X, L_{p})$, by writing $\tilde{v}=e^{2u_{p}}v+ s$ for some $s \in H^{0}(X, L_{p})$, one can see
\begin{equation}\label{ineq1}
 (e^{2u_{p}}v+s, e^{2u_{p}}v+s)_{g_{p}}=\int_{X}{|\tilde{v}|^{2}_{g_{p}}\frac{\omega^{n}}{n!}} =\int_{X}{(|e^{2u_{p}}v|^{2}_{g_{p}}+ |s|^{2}_{g_{p}})\frac{\omega^{n}}{n!}} \geq \int_{X}{|e^{2u_{p}}v|^{2}_{g_{p}}\frac{\omega^{n}}{n!}}.
 \end{equation}From (\ref{ineq}) and (\ref{ineq1}), we have \begin{equation}\label{ineq2}
                                                             \int_{X}{|e^{2u_{p}}v|^{2}_{g_{p}}\frac{\omega^{n}}{n!}} \leq \int_{X}{|\tilde{v}|^{2}_{g_{p}} \frac{\omega^{n}}{n!}} \leq \frac{2}{A_{p}}\int_{X}{|\beta|^{2}_{g_{p}}\frac{\omega^{n}}{n!}}.
                                                            \end{equation}
Let us now estimate $|\beta|^{2}_{g_{p}}$ from above.  By (\ref{beta}) and the first upper bound in (\ref{L22}), we obtain the following \begin{equation}\label{ineq3}
  |\beta|^{2}_{g_{p}}=e^{2u_{p}}|2 \overline{\partial}u_{p} \wedge v + \overline{\partial}v|^{2}_{h_{p}} \leq 2 e^{2u_{p}}(4|\overline{\partial}u_{p} \wedge v|^{2}_{h_{p}} + |\overline{\partial} v|^{2}_{h_{p}}) \leq 2 e^{2u_{p}}(\frac{A_{p}}{16} |v|^{2}_{h_{p}} +|\overline{\partial}v|^{2}_{h_{p}}),
   \end{equation}where, in the first estimation, we use an elementary inequality for norms: $|x+y|^{2} \leq 2(|x|^{2}+|y^{2}|)$.
 Finally, putting (\ref{ineq3}) into (\ref{ineq2}) finishes the proof.
\end{proof}

\section{Bergman Kernel Estimations}

\begin{thm}\label{bkad}
Let $(X, \omega)$ be a compact K\"{a}hler manifold with $\dim_{\mathbb{C}}{X}=n$. Let $\{(L_{p}, h_{p})\}^{\infty}_{p\geq 1}$ be a sequence of holomorphic line bundles with Hermitian metrics $h_{p}$ of class $\mathscr{C}^{3}$ such that (\ref{appr}) holds. Assume that $\eta_{p}=\frac{\|h_{p}\|^{1/3}_{3}}{\sqrt{A_{p}}} \rightarrow 0$ as $p \rightarrow \infty$. Then we have \begin{equation}\label{foa}
                                                                                                  \lim_{p\rightarrow \infty}\frac{K_{p}(x)}{A^{n}_{p}}=1.
                                                                                                \end{equation}
\end{thm}

\begin{proof}
We begin by taking a reference cover of the K\"{a}hler manifold $X$, as in Definition \ref{refed}. Selecting $x\in X$ and a corresponding  $z$-coordinate system based on Lemma \ref{refe} at $x\in X$. Then

\begin{align}\label{lwe}
\varphi_{p}(z) &= \Re{t_{p}(z)} + \varphi'_{p}(z) + \widetilde{\varphi_{p}}(z), \quad \varphi'_{p}(z) = \sum_{j=1}^{n} \lambda^{p}_{j} |z_{j}|^2,
\end{align} $\varphi_{p}$ is the weight for the Hermitian metric $h_{p}$ on $P^{n}(x, R)$ satisfying the condition (ii) in Lemma \ref{refe} and $t_{p}$ is the polynomial of degree at most $2$. Let $e_{p}$ be a local frame of $L_{p}$ on $U_{j}$ with the norm $|e_{p}|_{h_{p}}= e^{-\varphi_{p}} $. Next, we choose $R_{p}\in (0, R/2)$, which we will determine later.

To estimate the norm of a section $S\in H^{0}(X, L_{p})$ at the point $x \equiv 0$, we consider $S=f\,e_{p}$, where $f$ is a holomorphic function on $P^{n}(x, R)$. Utilizing the sub-averaging property for plurisubharmonic functions, we obtain:

\begin{align}\label{subav}
|S(x)|^2_{h_p} &= |f(0) e^{-t_{p}(0)}|^2= |f(0)|^{2} e^{-2\Re{t_{p}(0)}} \leq \frac{\int_{P^{n}(0,R_p)} {|f|^2 e^{-2\Re{t_{p}}} e^{-2\varphi'_{p}} e^{-2 \widetilde{\varphi_{p}}} \frac{\omega^{n}}{n!}}}{ \int_{P^{n}(0,R_p)} {e^{-2\varphi'_{p}} e^{-2 \widetilde{\varphi_{p}}} \frac{\omega^{n}}{n!} }}.
\end{align}

For the right-hand side of (\ref{subav}), by Lemma \ref{refe}, there exists a constant $D>0$ such that $ -D \|h_{p}\|_{3} |z|^{3} \leq \widetilde{\varphi_{p}}(z) \leq D \|h_{p}\|_{3} |z|^{3}$, and by considering (\ref{innp}) and (\ref{lwe}) we have,

\begin{align*} \label{4.4}
 \frac{\int_{P^{n}(0,R_p)} {|f|^2 e^{-2\Re{t_{p}}} e^{-2\varphi'_{p}} e^{-2 \widetilde{\varphi_{p}}} \frac{\omega^{n}}{n!}}}{ \int_{P^{n}(0,R_p)} {e^{-2\varphi'_{p}} e^{-2 \widetilde{\varphi_{p}}} \frac{\omega^{n}}{n!} }}  & \leq \frac{ \int_{P^{n}(0,R_p)}{|f|^2e^{-2\varphi_{p}} \frac{\omega^{n}}{n!}}}{\int_{P^{n}(0, R_{p})}{e^{-2\varphi'_{p}} e^{-2\widetilde{\varphi_{p}}} \frac{dV}{1+D\,R^{2}_{p}} }} \nonumber\\
                                                                              & \leq \frac{(1 + D\,R^{2}_{p}) e^{2\,D\, ||h_{p}||_{3} R^{3}_{p}} \,\|S\|^{2}_{p}}{\int_{P^{n}(0, R_{p})}{e^{-2 \, \varphi'_{p}} dV }} .
\end{align*}Combining the above inequality with (\ref{subav}) yields

\begin{equation}\label{fest}
  |S(x)|^2_{h_p} = |f(0) e^{-t_{p}(0)}|^2= |f(0)|^{2} e^{-2\Re{t_{p}(0)}} \leq \frac{(1 + D\,R^{2}_{p}) e^{2\,D\, ||h_{p}||_{3} R^{3}_{p}} \,\|S\|^{2}_{p}}{\int_{P^{n}(0, R_{p})}{e^{-2 \, \varphi'_{p}} dV }}.
\end{equation}

Let us now estimate the integral in the denominator of (\ref{fest}). To do this, we consider the Gaussian-type integrals of finite radius, \begin{equation}\label{finga}
             F(\rho):=\int_{|\xi| \leq \rho}{e^{-2|\xi|^{2}}dm(\xi)}=\frac{\pi}{2}(1- e^{-2\rho^{2}}),
           \end{equation}where $dm$ is the Lebesgue measure on $\mathbb{C}$. It is easy to see that $F$ is an increasing function of $\rho$. We also write \begin{equation}\label{Cn}
             F(\infty):=\lim_{\rho \rightarrow \infty}\int_{|\xi| \leq \rho}{e^{-2|\xi|^{2}} dm(\xi)}:= \int_{\mathbb{C}}{e^{-2|\xi|^{2}} dm(\xi)}=\frac{\pi}{2}.
           \end{equation}
Since  \begin{equation}\label{gaun}
         \int_{P^{n}(0, R_{p})}{e^{-2\varphi'_{p}} dV}= \int_{P^{n}(0, R_{p})}{e^{-2 (\lambda^{p}_{1}|z_{1}|^{2} + \ldots + \lambda^{p}_{n}|z_{n}|^{2} )} dV(z)},
       \end{equation} it is enough to treat the integral
       \begin{equation}\label{gaun1}
         \int_{\Delta(0, R_{p})}{e^{-2 \lambda^{p}_{j}|z_{j}|^{2}}} dm(z_{j})
       \end{equation}
       in order to get a lower bound for the integral (\ref{gaun}), where $\Delta(0, R_{p})$ is the unit closed disk in $\mathbb{C}$. By the relation (\ref{diops}), there exists $p_{1} \in \mathbb{N}$ such that, for all $p>p_{1}$, \begin{equation}\label{diop1}
                                                                                                              \frac{3A_{p}}{4} \omega_{x} \leq c_{1}(L_{p}, h_{p})_{x} \leq \frac{5A_{p}}{4} \omega_{x},
                                                                                                            \end{equation}
       which, on account of (\ref{kahloc}) and (\ref{locchern}), leads to
        \begin{equation} \label{dioplambda}
       \frac{3 \pi A_{p}}{8} \leq \lambda^{p}_{j} \leq \frac{5 \pi A_{p}}{8}.
                                                  \end{equation}
 Let us go back to the integral (\ref{gaun}),
 \begin{equation*}
         \int_{\Delta(0, R_{p})=\{|z_{j}| \leq R_{p}\}}{e^{-2 \lambda^{p}_{j}|z_{j}|^{2}}} dm(z_{j}),
       \end{equation*}which, by a change of variable $\big(\sqrt{{\lambda_{j} ^{p}}} z_j= w_j \big)$, equals the following \begin{equation*}
                                                                             \frac{1}{\lambda^{p}_{j}} \int_{\{|w_{j}|\leq R_p \sqrt{\lambda_{j}^p}\}}{e^{-2 |w_{j}|^{2}}} dm(w_{j}).
                                                                           \end{equation*} Now, (\ref{dioplambda}) gives $\sqrt{\frac{\lambda^{p}_{j}}{A_{p}}} > \sqrt{\frac{3\pi}{8}}>1$, which gives $R_{p} \sqrt{A_{p}} \leq R_{p} \sqrt{\lambda^{p}_{j}}$. Combining this with the fact that $F$ is increasing, we get
       \begin{equation}\label{lambdas}
       \int_{\Delta(0, R_{p})=\{|z_{j}| \leq R_{p}\}}{e^{-2 \lambda^{p}_{j}|z_{j}|^{2}}} dm(z_{j})= \frac{1}{\lambda^{p}_{j}} \int_{\big\{|w_{j}| \leq R_{p} \sqrt{\lambda^{p}_{j}}\big\}}{e^{-2 |w_{j}|^{2}}} dm(w_{j}) \geq \frac{F(R_{p} \sqrt{A_{p}})}{\lambda^{p}_{j}}.
       \end{equation} Consequently, from (\ref{gaun}), we have \begin{equation*}\label{lambdapro}
                                                                \int_{P^{n}(0, R_{p})}{e^{-2\varphi'_{p}} dV} \geq \frac{F(R_{p} \sqrt{A_{p}})^{n}}{ \lambda^{p}_{1} \ldots \lambda^{p}_{n}}.
                                                               \end{equation*}Inserting this last inequality in (\ref{fest}) give   \begin{equation}\label{vars}
                                                                                                                               |S(x)|^{2}_{h_{p}} \leq \frac{(1+D R_{p}^{2}) e^{2D \,\|h_{p}\|_{3} \, R^{3}_{p}}}{F(R_{p} \sqrt{A_{p}})^{n}}\lambda^{p}_{1} \ldots \lambda^{p}_{n} \,\|S\|_{p}^{2}.
                                                                                                                             \end{equation} If we take the supremum in (\ref{vars}) for all $S \in H^{0}(X, L_{p})$ with $\|S\|_{p}=1$ and use the variational principle (\ref{variat}) for $K_{p}$, we get \begin{equation}\label{estad}
                                                                                                                               K_{p}(x) \leq \frac{(1+D R_{p}^{2}) e^{2D \,\|h_{p}\|_{3} \, R^{3}_{p}}}{F(R_{p}                                                                                                \sqrt{A_{p}})^{n}}\lambda^{p}_{1} \ldots \lambda^{p}_{n}
                                                                                                                             \end{equation}for any $R_{p} \in (0, \frac{R}{2})$.

       We will now determine a lower bound for $K_{p}$ by employing $L^{2}$ estimations obtained earlier as Theorem \ref{De3}. Let $\kappa: \mathbb{C}^{n} \rightarrow [0, 1]$ be a cut-off function with a compact support in $P^{n}(0, 2)$, $\kappa=1$ on $P^{n}(0, 1)$. By defining $\kappa_{p}(z):=\kappa(\frac{z}{R_{p}})$, we consider $H=\kappa_{p}\, e^{t_{p}} e_{p}$, which is a (smooth) section of $L_{p}$, and $|H(x)|^{2}_{h_{p}}=|\kappa_{p}(x)|^{2}e^{2\Re{t_{p}}(x)}\, e^{-\varphi_{p}(x)}$. We estimate $\|H\|_{p}$ from above as follows:

       \begin{equation}\label{dest1a}
       \|H\|^{2}_{p} \leq \int_{P^{n}(0, 2R_{p})} e^{2\Re{t_{p}(x)}}{e^{-2 \varphi_{p}(x)} \frac{\omega_{x}^{n}}{n!}} = \int_{P^{n}(0, 2R_{p})} e^{-2\varphi_{p}'(x)}e^{-2\widetilde{\varphi}_{p}(x)}\frac{\omega_{x}^{n}}{n!}
       \end{equation} By using Lemma \ref{refe} (i) along with the relations (\ref{Cn}) and (\ref{lambdas}) on the integral at the very right end of the inequality (\ref{dest1a}), we get the following

\begin{equation} \label{H}
 \|H\|^{2}_{p} \leq (1 + 4D R^{2}_{p}) e^{16 D \|h_{p}\|_{3}\,R^{3}_{p}} \,\int_{P^{n}(0, 2R_{p})}{e^{-2\varphi'_{p}} dV}\leq(1 + 4D\, R^{2}_{p}) e^{16\, D\, \|h_{p}\|_{3}\,R^{3}_{p}}\, (\frac{\pi}{2})^{n}\, \frac{1}{\lambda^{p}_{1} \ldots \lambda^{p}_{n}}.
\end{equation}

Let us define $\Phi = \overline{\partial} H$. Noting that $\|\overline{\partial} \kappa_{p}\|^{2} = \|\overline{\partial} \kappa \|^{2} / R_{p}^{2}$, where $\|\overline{\partial} \kappa\|$ is the supremum of $|\overline{\partial} \kappa|$, we deduce the following inequality:

\begin{equation*}
\|\Phi\|_p^2 \leq \int_{P^{n}(0,2R_{p})} {|\overline{\partial}\kappa_{p}|^{2} e^{-2\varphi'_{p}}e^{-2\widetilde{\varphi}_{p}} \frac{\omega^{n}}{n!}} \leq \frac{\|\overline{\partial}\kappa\|^{2}}{R_{p}^{2}} (\frac{\pi}{2})^{n} \frac{(1 + 4D\, R^{2}_{p}) e^{16\, D\, \|h_{p}\|_{3}\,R^{3}_{p}}}{\lambda^{p}_{1} \ldots \lambda^{p}_{n}}.
\end{equation*}

As $A_{p} \rightarrow \infty$, by using Theorem \ref{De3}, there exists $p>p_{0}$ such that, for all $p > p_0$, we can find a smooth section $\Gamma$ of $L_{p}$ as a solution to the $\overline{\partial}$-equation for $\Phi$ such that $\overline{\partial} \Gamma=\Phi= \overline{\partial} H$ and
\begin{equation}\label{Gp} \|\Gamma\|_{p}^{2} \leq \frac{2}{A_{p}} \|\Phi\|_{p}^{2} \leq \frac{2\|\overline{\partial}\kappa\|}{A_{p}\,R_{p}^{2}} (\frac{\pi}{2})^{n} \frac{( 1 + 4D\,R_{p}^2)\,e^{(16D \,\|h_{p}\|_{3} R_{p}^{3})}\,}{\lambda^{p}_{1} \ldots \lambda^{p}_{n}}.
\end{equation}

Given that $H=e_{p}$ is holomorphic on $P^{n}(0, R_{p})$, $\Gamma$ is holomorphic on $P^{n}(0, R_{p})$ as well since $\overline{ \partial }\Gamma= \overline{ \partial }H=0$ on $P^{n}(0, R_{p})$. Applying estimate (\ref{vars}) to $\Gamma$ on $P^{n}(0, R_{p})$ leads us to the following inequality

\begin{align}\label{gs}
  |\Gamma(x)|^{2}_{h_{p}}  & \leq \frac{(1+D\,R^{2}_{p}) e^{2D \|h_{p}\|_{3} R^{3}_{p}}}{F(R_{p} \sqrt{A_{p}})^{n}} \lambda^{p}_{1} \ldots \lambda^{p}_{n} \|G\|^{2}_{p}
 \\
   & \leq \frac{2\|\overline{\partial } \kappa\|^{2}}{A_{p} R^{2}_{p} F(R_{p} \sqrt{A_{p}})^{n}} (\frac{\pi}{2})^{n} \,(1+4\,D\,R^{2}_{p})^{2}\,e^{18\,D\,\|h_{p}\|_{3}\, R^{3}_{p}}.
\end{align}

Now we will construct a new section $\Lambda:=H- \Gamma \in H^{0}(X, L_{p})$. Then, by a basic inequality $|S_{1}-S_{2}|^{2}_{h_{p}} \geq (|S_{1}|_{h_{p}} -|S_{2}|_{h_{p}})^{2}$ for norms applied to $\Lambda$, combined with (\ref{gs}) and the observation $|F(x)|_{h_{p}}=1$, we get \begin{align}\label{ssecp}
                                                                                |\Lambda(x)|^{2}_{h_{p}} & \geq (|H(x)|_{h_{p}}- |\Gamma(x)|_{h_{p}})^{2} \nonumber \\
                                                                                 & \geq \big( 1- (\frac{\pi}{2})^{n/2} \frac{\sqrt{2} \|\overline{\partial} \kappa\| (1+ 4 \,D\,R^{2}_{p})}{R_{p} \sqrt{A_{p}} F(R_{p} \sqrt{A_{p}})^{n/2}} e^{9 \, D\,\|h_{p}\|_{3}\,R^{3}_{p}}   \big)^{2}.
                                                                              \end{align}

On the other hand, by (\ref{H}) and (\ref{Gp}) together with the triangle inequality, we obtain    \begin{equation}\label{asd}
            \|\Lambda\|^{2}_{p} \leq (\|H\|_{p}+ \|\Gamma\|_{p})^{2} \leq (\frac{\pi}{2})^{n} \frac{1}{\lambda^{p}_{1} \ldots \lambda^{p}_{n}} (1+ 4\,D\,R^{2}_{p})\,e^{16\,D\,\|h_{p}\|_{3}\,R^{3}_{p}} \big(1+ \frac{\sqrt{2} \|\overline{\partial} \kappa\|}{R_{p} \sqrt{A_{p}}}\big)^{2}.
          \end{equation} To simplify what we have done so far, we write \begin{equation}\label{B1}
                                    B_{1}(R_{p}):=\big( 1- (\frac{\pi}{2})^{n/2} \frac{\sqrt{2} \|\overline{\partial} \kappa\| (1+ 4 \,D\,R^{2}_{p})}{R_{p} \sqrt{A_{p}} F(R_{p} \sqrt{A_{p}})^{n/2}} e^{9 \, D\,\|h_{p}\|_{3}\,R^{3}_{p}}   \big)^{2}
                                  \end{equation} and \begin{equation}\label{B2}
                                                        B_{2}(R_{p}):=  (1+ 4\,D\,R^{2}_{p})\,e^{16\,D\,\|h_{p}\|_{3}\,R^{3}_{p}} \big(1+ \frac{\sqrt{2} \|\overline{\partial} \kappa\|}{R_{p} \sqrt{A_{p}}}\big)^{2}.
                                                      \end{equation}
The variational property (\ref{variat}) combined with (\ref{ssecp}) and (\ref{asd}) implies

\begin{equation}\label{lbb}
 K_{p}(x) \geq \frac{|\Lambda(x)|^{2}_{h_{p}}}{\|\Lambda\|^{2}_{p}} \geq \frac{\lambda^{p}_{1} \ldots  \lambda^{p}_{n}}{(\frac{\pi}{2})^{n}}. \frac{B_{1}(R_{p})}{B_{2}(R_{p})}.
\end{equation}
For the upper bound (\ref{estad}), as above, we put \begin{equation}\label{B3}
         B_{3}(R_{p}):=(\frac{\pi/2}{F(R_{p} \sqrt{A_{p}})})^{n} (1+ D\,R^{2}_{p})\,e^{2\,D\,\|h_{p}\|_{3}\,R^{3}_{p}}.\,
       \end{equation}Observe that \begin{equation}\label{B3u}
                                    (\frac{\pi}{2})^{n} K_{p}(x) \leq B_{3}(R_{p})\,\lambda^{p}_{1} \ldots \lambda^{p}_{n}.
                                  \end{equation}By our hypothesis $\eta_{p}=\frac{\|h_{p}\|^{1/3}_{3}}{\sqrt{A_{p}}} \rightarrow 0$, and now we determine $R_{p}$ in the following way  $$R_{p}:=\eta_{p}^{1/3} \|h_{p}\|^{-1/3}_{3}=\frac{\eta^{-2/3}_{p}}{\sqrt{A_{p}}},$$ which means $$\eta_{p}=\|h_{p}\|_{3} R^{3}_{p},\,\, \eta^{-2/3}_{p}=R_{p} \sqrt{A_{p}}.$$ Since $\|h_{p}\|_{3} \geq 1$ from Subsection \ref{metricss}, we have $R_{p} \leq \eta^{1/3}_{p}$, and so $R_{p} \rightarrow 0$ when $p \rightarrow \infty$. All in all, based on $R_{p}$, it follows from the quantities $B_{1}(R_{p}), \,B_{2}(R_{p})$ and $B_{3}(R_{p})$ that we find uniform upper and lower bounds for $K_{p}$ depending only on $\eta_{p}$ \begin{equation}\label{uniformp}
                                                                                                    \frac{B_{1}(R_{p})}{B_{2}(R_{p})} \geq 1- D' \eta^{2/3}_{p}\,\, \text{and}\,\,B_{3}(R_{p}) \leq 1+ D' \eta^{2/3}_{p}.
                                                                                                  \end{equation}Here $D' >0$ denotes a constant that merely depends on the reference cover. We finally consider the following inequality that holds for all $p > p_{0}$ \begin{equation}\label{feb}
                                                                                                                  \frac{1}{A^{n}_{p}} (\frac{\pi}{2})^{n}  \frac{\lambda^{p}_{1} \ldots  \lambda^{p}_{n}}{(\frac{\pi}{2})^{n}}. \frac{B_{1}(R_{p})}{B_{2}(R_{p})}   \leq  (\frac{\pi}{2})^{n} \frac{K_{p}(x)}{A^{n}_{p}} \leq \frac{1}{A^{n}_{p}}B_{3}(R_{p}) \lambda^{p}_{1} \ldots \lambda^{p}_{n},
                                                                                                                       \end{equation} which, in light of the findings (\ref{limitl1}), (\ref{lbb}), (\ref{B3u}) and (\ref{uniformp}), finishes the proof.

\end{proof}

Relying on the proof in \cite{BCM}, which incorporates methods from \cite{Brn}, \cite{CMM} and \cite{Li}, we provide a proof for the off-diagonal decay estimate of the Bergman kernels $K_{p}(x, y)$ associated with the corresponding line bundles $(L_{p}, h_{p})$ in the current setting of diophantine approximation.

\begin{thm}\label{udeb}
Let $(X, \omega)$ be a compact K\"{a}hler manifold with $\dim_{\mathbb{C}}{X}=n$. Let $\{(L_{p}, h_{p})\}_{p\geq 1}$ be a sequence of holomorphic line bundles with Hermitian metrics $h_{p}$ of class $\mathscr{C}^{3}$ such that (\ref{appr}) is satisfied. Write $\eta_{p}= \frac{\|h_{p}\|^{1/3}_{3}}{\sqrt{A_{p}}} \rightarrow 0$ when $p\rightarrow \infty$. Then there exist constants $G, B >0$, $p_{0} \geq 1$ such that  for every $x, y \in X$ and $p >p_{0}$, the following estimation holds true \begin{equation}\label{BOD}
                       |K_{p}(x, y)|^{2}_{h_{p}} \leq G e^{-B \sqrt{A_{p}}d(x, y)} A^{2n}_{p}.
                     \end{equation}
\end{thm}

\begin{proof}

Initially, we select a reference cover for $X$ in accordance with the earlier definition above and choose a large enough $p_{0} \in \mathbb{N}$ such that \begin{equation*}
                                          R_{p}:=\frac{1}{\sqrt{A_{p}}} < \frac{R}{2}
                                        \end{equation*} and Theorem \ref{De3} and Theorem \ref{De4} are valid for all $p>p_{0}$.

Let $y \in X$ and $r > 0$. Write
\begin{equation*} B(y,r) := \{ x \in X : d(y,x) < r \}, \end{equation*} which is the ball of radius $r>0$ centered at $y$. Choose a constant $\theta \geq 1$ so that for any $ y \in X$,
$$P^n(y, R_p) \subseteq B(y, \theta R_p),$$ where
$P^n(y, R_p)$ is the (closed) polydisk centered at $y$ given by the coordinates centered at $y$ in view of Lemma \ref{refe}.

\underline{\textit{Claim:}} There exists a constant $D' >1$ such that if $y \in X$, so $y\in P^{n}(x_{j}, R_{j})$ for some $j$ and $z$-coordinates centered at $y$ are due to Lemma \ref{refe}, then \begin{equation}\label{subh}
                                                                                     |S(y)|^{2}_{h_{p}} \leq D' \,A_{p}^{n} \int_{P^{n}(y, R_{p})}{|S|_{h_{p}}^{2} \frac{\omega^{n}}{n!}},
                                                                                   \end{equation} where, as above, $P^{n}(y, R_{p})$ is the (closed) polydisk centered at $y\equiv 0$ in the $z$-coordinates and $S$ is an arbitrary continuous section of $L_{p}$ on $X$ which is holomorphic on $P^{n}(y, R_{p})$.

\textit{\underline{Proof of the Claim:}} By Lemma \ref{refe}(ii), $(L_{p}, h_{p})$ has a weight $\varphi_{p}$ on $P^{n}(y, R)$ such that \begin{equation}\label{wevar}
                          \varphi_{p}(z)= \Re{t_{p}(z)} + \varphi'_{p}(z) + \widetilde{\varphi_{p}}(z),
                        \end{equation}where $t_{p}(z)$ is, as in Lemma \ref{refe}, a holomorphic polynomial of degree at most $2$, $\varphi'_{p}(z)= \sum_{l=1}^{n}{\lambda^{p}_{l} |z_{l}|^{2}}$ , and $\widetilde{\varphi_{p}}(z)$ satisfies the inequality \begin{equation}\label{tilde}-D \|h_{p}\|_{3} |z|^{3} \leq \widetilde{\varphi_{p}}(z)  \leq D \|h_{p}\|_{3} |z|^{3}\end{equation} for $z\in P^{n}(y, R)$ (Recall that $R=\min{R_{j}}$). Let $e_{p}$ be a frame of $L_{p}$ on $U_{j}$ so that $S=f\,e_{p}$, where $f$ is a holomorphic function on $P^{n}(y, R_{p})$ and $|e_{p}|_{h_{p}}=e^{-\varphi_{p}}$. As in the beginning of the proof of Theorem \ref{bkad}, we have first the relation, which is nothing but (\ref{fest})

\begin{equation}\label{lest}
  |S(y)|^2_{h_p} = |f(0) e^{-t_{p}(0)}|^2= |f(0)|^{2} e^{-2\Re{t_{p}(0)}} \leq \frac{(1 + D\,R^{2}_{p}) e^{2\,D\, ||h_{p}||_{3} R^{3}_{p}} \,\|S\|^{2}_{p}}{\int_{P^{n}(0, R_{p})}{e^{-2 \, \varphi'_{p}} dV }}.
\end{equation}
Since  \begin{equation}\label{gaunnn}
         \int_{P^{n}(0, R_{p})}{e^{-2\varphi'_{p}} dV}= \int_{P^{n}(0, R_{p})}{e^{-2 (\lambda^{p}_{1}|z_{1}|^{2} + \ldots + \lambda^{p}_{n}|z_{n}|^{2} )} dV(z)}
       \end{equation}as we have done in the proof of Theorem \ref{bkad}, it will be sufficient for us to find a lower bound for the integral
       \begin{equation}\label{gaun1}
         \int_{\Delta(0, R_{p})}{e^{-2 \lambda^{p}_{j}|z_{j}|^{2}}} dm(z_{j})
       \end{equation}
       in order to get a lower bound for the whole integral (\ref{gaunnn}), where $\Delta(0, R_{p})$ is the unit closed disk in $\mathbb{C}$. By the relation (\ref{dioplambda}), there exists $p_{1} \in \mathbb{N}$ such that, for all $p>p_{1}$, $\frac{3 \pi A_{p}}{8} \leq \lambda^{p}_{j} \leq \frac{5 \pi A_{p}}{8},$ and as before, $\sqrt{\frac{\lambda^{p}_{j}}{A_{p}}} > \sqrt{\frac{3\pi}{8}}>1$. We also observe that \begin{equation}\label{FF}
                                                                                            F(1)=\frac{\pi}{2}(1-\frac{1}{e^{2}})> \frac{\pi}{2}\, \frac{21}{25}>1
                                                                                          \end{equation} since $\frac{5}{2} < e < 3$. By the same argument used in the proof of Theorem \ref{bkad}, we get
       \begin{equation*}
       \int_{\Delta(0, R_{p})=\{|z_{j}| \leq \frac{1}{\sqrt{A_{p}}}\}}{e^{-2 \lambda^{p}_{j}|z_{j}|^{2}}} dm(z_{j})= \frac{1}{\lambda^{p}_{j}} \int_{\big\{|w_{j}| \leq \sqrt{\frac{\lambda^{p}_{j}}{A_{p}}}\big\}}{e^{-2 |w_{j}|^{2}}} dm(w_{j}) \geq \frac{1}{\lambda^{p}_{j}},
       \end{equation*} since $F(1)>1$ by (\ref{FF}) and $F$ is increasing.  Consequently, from (\ref{gaunnn}), we have \begin{equation}\label{lambdapro}
                                                                \int_{P^{n}(0, R_{p})}{e^{-2\varphi'_{p}} dV} \geq \frac{1}{ \lambda^{p}_{1} \ldots \lambda^{p}_{n}}.
                                                               \end{equation}Inserting (\ref{lambdapro}) into the sub-mean estimation (\ref{lest}) and using (\ref{dioplambda}), one has
       \begin{align}\label{est2}
         |S(y)|_{h_{p}}^{2} & \leq (1+ D R^{2}_{p}) e^{2D\|h_{p}\|_{3} R^{3}_{p}} \lambda^{p}_{1} \ldots \lambda^{p}_{n} \int_{P^{n}(y, R_{p})}{|S|^{2}_{h_{p}} \frac{\omega^{n}}{n!}}\\
          & \leq (1+ D R^{2}_{p}) e^{2D\|h_{p}\|_{3} R^{3}_{p}} (\frac{5 \pi}{8})^{n} A^{n}_{p} \int_{P^{n}(y, R_{p})}{|S|^{2}_{h_{p}} \frac{\omega^{n}}{n!}}.
       \end{align} As $R_{p}=\frac{1}{\sqrt{A_{p}}}  \rightarrow 0$ and by our assumption that $\eta_{p}=\frac{\|h_{p}\|^{1/3}_{3}}{\sqrt{A_{p}}} \rightarrow 0$ when $p \rightarrow \infty$, one can find a constant $D' >1$ such that, for a large enough $p_{2} \in \mathbb{N}$, \begin{equation*}
                                                                           (1+ D R^{2}_{p}) e^{2D\|h_{p}\|_{3} R^{3}_{p}} (\frac{5 \pi}{8})^{n} \leq D'
                                                                         \end{equation*} for all $p > \max{\{p_{0}, p_{1}, p_{2}\}}$. Hence \begin{equation}\label{est3}
                                                                        |S(y)|_{h_{p}}^{2} \leq D'\, A^{n}_{p} \int_{P^{n}(y, R_{p})}{|S|^{2}_{h_{p}} \frac{\omega^{n}}{n!}},
                                                                         \end{equation} which completes the proof of the claim.
 Let us fix $x\in X$. Then there exists $S_{p}=S_{p, x} \in H^{0}(X, L_{p})$ such that \begin{equation*}
                                                                                         |S_{p}(y)|^{2}_{h_{p}}=|K_{p}(x, y)|^{2}_{h_{p}}                                                                                       \end{equation*} for all $y \in X$.
By Theorem \ref{bkad}, there exists a constant $D'' >1$ and $p_{3} \in \mathbb{N}$ such that, for all $p > p_{3}$, \begin{equation}\label{ble}
                                                                                                                    K_{p}(x) \leq D'' A^{n}_{p},
                                                                                                                  \end{equation} where $D''$ is some constant that depends only on the reference cover.
On the other hand, \begin{equation}\label{normb}
                   \|S_{p}\|^{2}_{p}=\int_{X}{|S_{p}(y)|^{2}_{h_{p}} \frac{\omega^{n}(y)}{n!}}=  \int_{X}{|K_{p}(x, y)|^{2}_{h_{p}} \frac{\omega^{n}(y)}{n!}}=K_{p}(x).
                   \end{equation} In the rest of the proof, we proceed with the near-diagonal and off-diagonal estimations of $K_{p}(x, y)$.

For the near-diagonal estimation, let $y\in X$ and $d(x, y) \leq \frac{4\theta}{\sqrt{A_{p}}}$. By the variational property (\ref{variat}) of $K_{p}(x)$, the inequality (\ref{ble}) and (\ref{normb}), we have
\begin{equation}\label{ve2}
  |K_{p}(x, y)|^{2}_{h_{p}}=|S_{p}(y)|^{2}_{h_{p}}  \leq K_{p}(y) \|S_{p}\|^{2}_{p} \leq K_{p}(x)\,K_{p}(y) \leq (D'')^{2} A^{2n}_{p} \leq e^{4\theta} (D'')^{2} A^{2n}_{p}\, e^{-\sqrt{A_{p}}d(x, y)} .
\end{equation}

We go on with the far off-diagonal estimation. Let $y\in X$, and this time, consider \begin{equation*}
                                                                                   \delta:=d(x, y)> 4\theta\,\frac{1}{\sqrt{A_{p}}}=4 \theta R_{p}.
                                                                                 \end{equation*} By the choice of $S_{p}$ and the claim in the beginning of the proof, we get
\begin{equation}\label{offb}
  |S_{p}(y)|^{2}_{h_{p}}=|K_{p}(x, y)|^{2}_{h_{p}} \leq A^{n}_{p} \int_{P^{n}(y, R_{p})}{|K_{p}(x, \zeta)|^{2}_{h_{p}} \frac{\omega^{n}_{\zeta}}{n!}}.
\end{equation} We observe that the inclusions \begin{equation}\label{inclus} P^{n}(x, R_{p}) \subset B(x, \frac{\delta}{4}) \,\,\text{and}\,\, P^{n}(y, R_{p}) \subset \{\zeta \in X: d(x, \zeta) >\frac{3\delta}{4}\} \end{equation} hold.

Let $\beta$ be a non-negative smooth function on $X$ with the following properties:
\begin{align*}
  \beta(\zeta)=1 & \,\,\text{if}\,\, d(x, \zeta) \geq \frac{3\delta}{4} \\
  \beta(\zeta)=0 & \,\,\text{if}\,\, d(x, \zeta) \leq \frac{\delta}{2} \\
  |\overline{\partial}\beta(\zeta)|^{2} \leq \frac{c}{\delta^{2}} \beta(\zeta) &\,\,\text{for}\,\,\text{some}\,\,c>0.
\end{align*}
According to these data, we first have
\begin{equation}\label{off2}
  \int_{P^{n}(y, R_{p})}{|K_{p}(x, \zeta)|^{2}_{h_{p}} \frac{\omega^{n}_{\zeta}}{n!}} \leq \int_{X}{|K_{p}(x, \zeta)|^{2}_{h_{p}} \beta(\zeta)\frac{\omega^{n}_{\zeta}}{n!}}.
\end{equation}
  Using the variational property for $K_{p}(x)$, the right-hand side of the inequality (\ref{off2}) takes the following form:
  \begin{equation*}
   \max\{|K_{p}(\beta S)|^{2}_{h_{p}}: S\in H^{0}(X, L_{p}), \,\int_{X}{|S|_{h_{p}}^{2}\, \beta \, \frac{\omega^{n}}{n!}=1}\},
  \end{equation*}where \begin{equation*}
                         K_{p}(\beta \,S)(x)=\int_{X}{K_{p}(x, \zeta) (\beta(\zeta)\,S(\zeta)) \frac{\omega^{n}_{\zeta}}{n!}}
                       \end{equation*} is the Bergman projection of the smooth section $\beta S$ to $H^{0}(X, L_{p})$. Note that when $\beta \equiv 1$, the usual variational formula (\ref{variat}) is obtained. Therefore, if we manage to estimate $|K_{p}(\beta S)|^{2}_{h_{p}}$, then we will be done. To this end, to find an upper bound for $|K_{p}(\beta S)|^{2}_{h_{p}}$, we use the decomposition of the space $L^{2}(X, L_{p})$ as below: \begin{equation}\label{decom}
                                                     L^{2}(X, L_{p})=H^{0}(X, L_{p}) \oplus Y.
                                                   \end{equation}(Since $L^{2}(X, L_{p})$ is an Hilbert space and $H^{0}(X, L_{p})$ is a closed subspace of it, such an orthogonal complementary subspace $Y$ always exists). Since $\beta S \in \mathscr{C}^{\infty}(X, L_{p}) \hookrightarrow L^{2}(X, L_{p})$, it follows from the decomposition (\ref{decom}) that there exists an element $u \in Y$ such that $$u= \beta S - K_{p}(\beta S).$$ Owing to the inclusion $P^{n}(x, R_{p}) \subset B(x, \frac{\delta}{4})$ from (\ref{inclus}) and $\beta(\zeta)=0$ for $\zeta \in B(x, \frac{\delta}{2})$ given previously in the proof, we readily have $\beta=0$ on $P^{n}(x, R_{p})$, so $u=\beta S - K_{p}(\beta S)=-K_{p}(\beta S)$, which is holomorphic on $P^{n}(x, R_{p})$ (defined by the coordinates centered at $x$ provided by Lemma \ref{refe}). Therefore, by using the claim in the beginning, one has
                                                   \begin{equation}\label{est3}
                                                     |K_{p}(\beta S)|^{2}_{h_{p}} = |u(x)|^{2}_{h_{p}} \leq D' A_{p}^{n} \int_{P^{n}(x, R_{p})}{|u|^{2}_{h_{p}} \frac{\omega^{n}}{n!}}.
                                                   \end{equation}We provide an upper bound for the integral on the right-hand side of (\ref{est3}) by Theorem \ref{De4}. For this purpose, let $\tau: [0, \infty) \rightarrow (-\infty, 0]$ be a smooth function defined as follows:
                                                   \begin{equation*}\tau(x):= \begin{cases}
                                                               0, & \mbox{if } x \leq \frac{1}{4} \\
                                                               -x, & \mbox{if}\,\, x \geq \frac{1}{2}.
                                                             \end{cases} \end{equation*}Write $\phi_{\delta}(x)=\delta \tau(\frac{x}{\delta})$. Note that $\tau'$ and $\tau''$ have compact supports within the set $[\frac{1}{4}, \frac{1}{2}]$, and so are $\phi'_{\delta}$ and $\phi''_{\delta}$. This means that there exists a constant $M_{0} >0$ such that $|\phi'_{\delta}(x)| \leq M_{0}$ and $|\phi''_{\delta}(x)| \leq \frac{M_{0}}{\delta}$ for all $x \geq 0$. Define the function \begin{equation*}
                                                                                                       v_{p}(\zeta):= \epsilon \sqrt{A_{p}} \phi_{\delta}(d(x, \zeta)).
                                                                                                     \end{equation*}Since $\phi'_{\delta}$ and $\phi''_{\delta}$ are smooth and have compact supports, we can find a constant $M_{1} >0$ such that \begin{align}\label{lastc}
                                                                                                       \|\overline{\partial}v_{p} \|_{L^{\infty}(X)} & \leq M_{1} \epsilon \sqrt{A_{p}}  \\
                                                                                                       dd^{c} v_{p} & \geq -\frac{M_{1}}{\delta} \epsilon \sqrt{A_{p}} \omega \geq - \frac{M_{1}}{4\theta} \epsilon A_{p} \omega
                                                                                                     \end{align} because of the inequality $\delta > 4\theta \frac{1}{\sqrt{A_{p}}}$. Now we can choose $\epsilon= \frac{1}{8M_{1}}$ for the conditions of Theorem \ref{De4} to hold. Since $\tau(x)=0$ for $x\leq \frac{1}{4}$, we have that $$v_{p}(\zeta)= \epsilon \,\sqrt{A_{p}}\, \phi_{\delta}(d(x, \zeta))= \epsilon\, \sqrt{A_{p}}\, \delta \,\tau(\frac{d(x, \zeta)}{\delta})=0$$ for $d(x, \zeta) \leq \frac{\delta}{4}$. Also, by the inclusion $P^{n}(x, R_{p}) \subset B(x, \frac{\delta}{4})$ given in (\ref{inclus}), we get $v_{p}(\zeta)= 0$ on $P^{n}(x, R_{p})$.
  By the definition of $\beta$, it is seen that $\overline{\partial} u=\overline{\partial} (\beta S)=\overline{\partial} \beta \wedge S$ (because $S$ is holomorphic) has the following (compact) support \begin{equation*}
                                                     U_{\delta}=\{\zeta \in X: \frac{\delta}{2} \leq d(x, \zeta) \leq \frac{3\delta}{4}\},
                                                   \end{equation*}so, for $\zeta \in U_{\delta}$, by the definitions of $U_{\delta}$, $\tau$ and $\phi_{\delta}$ we obtain \begin{equation*}
                                                                                                             v_{p}(\zeta)= \epsilon\, \sqrt{A_{p}} \, \phi_{\delta}(d(x, \zeta))= \epsilon \, \sqrt{A_{p}} \, \delta \, \tau(\frac{d(x, \zeta)}{\delta})= - \epsilon\, \sqrt{A_{p}}\, d(x, \zeta) \leq - \epsilon \sqrt{A_{p}} \,\frac{\delta}{2}.
                                                                                                           \end{equation*}
  By Theorem \ref{De4} and the definition of $\beta$, we have

  \begin{align*}
    \int_{P^{n}(x, R_{p})}{|u|^{2}_{h_{p}} \frac{\omega^{n}}{n!}} & \leq \int_{X}{|u|^{2}_{h_{p}} e^{v_{p}} \frac{\omega^{n}}{n!}} \\
      & \leq \frac{16}{3A_{p}} \int_{U_{\delta}}{|\overline{ \partial}(\beta S)|^{2}_{h_{p}} e^{v_{p}}\frac{\omega^{n}}{n!} } \\
      & \leq \frac{16c}{3A_{p} \delta^{2}} e^{-\epsilon \sqrt{A_{p}} \delta} \int_{U_{\delta}}{|S|^{2}_{h_{p}} \beta \frac{\omega^{n}}{n!} }\\
      & \leq \frac{c}{3} e^{-\sqrt{A_{p}} \delta} \leq c e^{-\sqrt{A_{p}} \delta}.
  \end{align*}Plugging this last inequality into (\ref{est3}) gives \begin{equation*}
                                                                        |K_{p}(\beta S)|^{2}_{h_{p}} \leq D'\, A^{n}_{p}\, c \, e^{-\epsilon \sqrt{A_{p}} \delta},
                                                                      \end{equation*}which gives, by using the inequality (\ref{off2}),
  \begin{equation*}
    \int_{P^{n}(y, R_{p})}{|K_{p}(x, \zeta)|^{2}_{h_{p}} \frac{\omega^{n}_{\zeta}}{n!}} \leq D' \,A^{n}_{p} \, c\, e^{-\epsilon \sqrt{A_{p}} d(x, y)}.
  \end{equation*}From the inequality (\ref{offb}), we infer \begin{equation}\label{lastbk}
                                                                       |K_{p}(x, y)|^{2}_{h_{p}} \leq c \, (D')^{2} A^{2n}_{p} \,e^{-\epsilon \sqrt{A_{p}}d(x, y)},
                                                                     \end{equation}which finalizes the proof.
\end{proof}

\subsection{Linearization} Let $V \subset X$, $U \subset \mathbb{C}^{n}$ be open subsets, and $x_{0} \in V$, $0 \in U$. Let us take a (K\"{a}hler) coordinate chart as follows $\gamma: (V, x_{0}) \rightarrow (U, 0)$, $\gamma(x_{0})=0$. We will use the following notation, so-called linearization of the coordinates on the K\"{a}hler manifold $X$: For any $u, v \in U \subset \mathbb{C}^{n}$, we write $\gamma^{-1}(u)=x_{0}+u$ and $\gamma^{-1}(v)=x_{0}+v$, and
\begin{equation}\label{lincha}
  K_{p}(\gamma^{-1}(u), \gamma^{-1}(v)):= K_{p}(x_{0}+ u, x_{0}+v).
\end{equation}
 Since $0 \in \mathbb{C}^{n}$ and $\mathbb{C}^{n}$ is a complex vector space, we can write $v=0+v$ and $u=0+u$. Linearization means that when we translate $0\in \mathbb{C}^{n}$ by $u$ (or by $v$ for that matter), by thinking of $0 \in \mathbb{C}^{n}$ as $x_{0}\in X$, we can also write $\gamma^{-1}(u)= x_{0}+ u$, so in local coordinates we express the difference between $\gamma^{-1}(u)$ and $x_{0}$ (not meaningful in $X$) by the difference $u-0$ (meaningful in $\mathbb{C}^{n}$ because $\mathbb{C}^{n}$ is a complex vector space). This is also called the abuse of notation in, for instance, \cite{SZ08} and \cite{SZ10}.

\vspace{10mm}
Fix a point $x$  in $X$ and let us take the coordinates centered at $x\equiv 0$, Kahler at $x$, as provided by Lemma $\ref{refe}$, with coordinate polydisk $P^n (x, R)$ for $R>0$. Let $R'>0$ be arbitrary and choose $p\gg1$ so that $R'/\sqrt{A_p}<R$. Modifying the argument in Theorem 2.3 of \cite{Bay17} (see also \cite{Ber}), we consider the following holomorphic functions
\begin{align*}
  \Gamma_{p}(u, v) &= \frac{K_{p}(\frac{u}{\sqrt{A_{p}}}, \frac{\overline{v}}{\sqrt{A_{p}}}) e^{-t_{p}(\frac{u}{\sqrt{A_{p}}}) -\overline{t_{p}(\frac{\overline{v}}{\sqrt{A_{p}}})}}   }{A^{n}_{p}\,\, e^{\frac{2}{A_{p}}\,\langle  \Lambda_{p} u, \overline{v} \rangle} \,\,}\, \\
    & =\frac{K_{p}(\frac{u}{\sqrt{A_{p}}}, \frac{\overline{v}}{\sqrt{A_{p}}}) e^{-t_{p}(\frac{u}{\sqrt{A_{p}}}) -\overline{t_{p}(\frac{\overline{v}}{\sqrt{A_{p}}})}} e^{-\widetilde{\varphi_{p}}(\frac{u}{\sqrt{A_{p}}})} e^{-\widetilde{\varphi_{p}}(\frac{\overline{v}}{\sqrt{A_{p}}})}  }{A^{n}_{p}\,\, e^{\frac{2}{A_{p}}\,\langle \Lambda_{p} u, \overline{v} \rangle} \,\,}\, e^{\widetilde{\varphi_{p}}(\frac{u}{\sqrt{A_{p}}})} e^{\widetilde{\varphi_{p}}(\frac{\overline{v}}{\sqrt{A_{p}}})}
\end{align*} on $\Omega=\{(u, v): u \in P^{n}(0, R'),\,\,v\in P^{n}(0, R')\} \subset \mathbb{C}^{n}_{u} \times \mathbb{C}^{n}_{v}$ in the respective coordinates $u, \,v$; $\Lambda_{p}:=Diag[\lambda^{p}_{1}, \ldots, \lambda^{p}_{n}]$, which is a diagonal matrix whose diagonal entries $\lambda^{p}_{j}$ are positive from the discussions in Section \ref{Sec2}. Let $\Omega_{0}:=\{(u, \overline{u}): u \in P^{n}(0, R')\}$. It follows from Theorem \ref{bkad} and Lemma \ref{refe} (ii) that $\Gamma_{p} \rightarrow 1$ on $\Omega_{0}$. Observe that since $\Gamma_{p}$ is uniformly bounded on $\Omega$, there is a subsequence $\{\Gamma_{p_{d}}\}$ such that $\Gamma_{p_{d}} \rightarrow \Gamma_{0}$ uniformly on $\Omega$, where we must have that $\Gamma_{0} \equiv 1$ on $\Omega_{0}$. Since $\Omega_{0}$ is a maximally totally real submanifold, we get $\Gamma_{0}=1$ on the whole $\Omega$. Since this argument can be applied to any subsequence of $\Gamma_{p}$, we see that $\Gamma_{p} \rightarrow 1$. We make now an observation for $|\Gamma_{p}(u, v)|^{2}$ that will be used in our main theorem. Since $\Lambda_{p}$ has positive diagonal entries, its square root $\Lambda_{p}^{1/2}$ is defined, so we have \begin{equation} \label{srd}\langle \Lambda_{p} u, \overline{v} \rangle= \langle \Lambda_{p}^{1/2} u, \Lambda_{p}^{1/2} \overline{v}  \rangle.\end{equation}

\begin{align}\label{odb}
  |\Gamma_{p}(u, v)|^{2} &= \frac{|K_{p}(\frac{u}{\sqrt{A_{p}}}, \frac{\overline{v}}{\sqrt{A_{p}}})|^{2} e^{-2 \Re{t_{p}(\frac{u}{\sqrt{A_{p}}})}} \,e^{-2 \Re{t_{p}(\frac{\overline{v}}{\sqrt{A_{p}}})}} }{A^{2n}_{p} \,e^{\frac{4}{A_{p}} \Re{ (\langle \Lambda_{p}u, \overline{v}  \rangle)}}} \\
   &= \frac{|K_{p}(\frac{u}{\sqrt{A_{p}}}, \frac{\overline{v}}{\sqrt{A_{p}}})|^{2} e^{-2 \Re{t_{p}(\frac{u}{\sqrt{A_{p}}})}} \,e^{-2 \Re{t_{p}(\frac{\overline{v}}{\sqrt{A_{p}}})}} }{A^{2n}_{p} \,e^{2 \Sigma_{j=1}^{n}{\frac{\lambda^{p}_{j}}{A_{p}}|u_{j}|^{2}}}\,e^{2 \Sigma_{j=1}^{n}{\frac{\lambda^{p}_{j}}{A_{p}}|v_{j}|^{2}}}\, e^{-2 \Sigma_{j=1}^{n}{\frac{\lambda^{p}_{j}}{A_{p}}|u_{j}- \overline{v_{j}}|^{2}}}} \nonumber\\
   &= \frac{|K_{p}(\frac{u}{\sqrt{A_{p}}}, \frac{\overline{v}}{\sqrt{A_{p}}})|^{2} e^{-2 \varphi_{p}(\frac{u}{\sqrt{A_{p}}})} \,e^{-2 \varphi_{p}(\frac{\overline{v}}{\sqrt{A_{p}}})} }{A^{2n}_{p} \,e^{-2 \Sigma_{j=1}^{n}{\frac{\lambda^{p}_{j}}{A_{p}}|u_{j} - \overline{v_{j}}|^{2}}}} e^{\widetilde{\varphi_{p}}(\frac{u}{\sqrt{A_{p}}})} e^{\widetilde{\varphi_{p}}(\frac{\overline{v}}{\sqrt{A_{p}}})} \nonumber,
\end{align}where, in the second equality, we have used (\ref{srd}) and the polarization identity $$\Re{(\langle \mathbf{x}, \mathbf{y} \rangle)}=\frac{1}{2}(\|\mathbf{x}\|^{2} + \|\mathbf{y}\|^{2} - \|\mathbf{x-y}\|^{2})$$ for the vectors $\mathbf{x}=\Lambda^{1/2}_{p} u$ and $\mathbf{y}=\Lambda^{1/2}_{p} \overline{v}$, and in the third equality, we take into account the representation $\varphi_{p}(z)=\Re{t_{p}(z)}+\sum_{j=1}^{m}{\lambda_{j}^{p}|z_{j}|^{2}}+ \widetilde{\varphi}_{p}(z)$ from Lemma \ref{refe} (ii). By the limit argument made above regarding the holomorphic functions $\Gamma_{p}$, the expression (\ref{odb}) for $|\Gamma_{p}(u, v)|^{2}$  and Lemma \ref{refe}(ii), we have \begin{equation*}\label{ofdb} \frac{|K_{p}(\frac{u}{\sqrt{A_{p}}}, \frac{\overline{v}}{\sqrt{A_{p}}})|^{2} e^{-2 \varphi_{p}(\frac{u}{\sqrt{A_{p}}})} \,e^{-2 \varphi_{p}(\frac{\overline{v}}{\sqrt{A_{p}}})} }{A^{2n}_{p} \,e^{-2 \Sigma_{j=1}^{n}{\frac{\lambda^{p}_{j}}{A_{p}}|u_{j} - \overline{v_{j}}|^{2}}}}  \rightarrow 1 \end{equation*} on $P^{n}(0, R')\times P^{n}(0, R')$ as $p\rightarrow \infty$. Therefore, we have proved the next theorem.

\begin{thm} \label{locunif}
    Let $(X, \omega)$ be a compact Kähler manifold of complex dimension $n$ and let a sequence of holomorphic line bundles $\{(L_{p}, h_{p})\}_{p \geq 1}$ be given, each equipped with a $\mathscr{C}^3$-class Hermitian metric $h_p$, satisfying the condition $(\ref{appr})$. Suppose that $\eta_{p}=\frac{\|h_{p}\|^{1/3}_{3}}{\sqrt{A_{p}}} \rightarrow 0$ as $p \rightarrow \infty$. Let $x$ be a fixed point in $X$. Then in  the local K\"{a}hler coordinates centered at  $x$ provided by Lemma \ref{refe}, we have

    \begin{equation}\label{linberg}
                 \frac{|K_{p}(\frac{u}{\sqrt{A_{p}}}, \frac{\overline{v}}{\sqrt{A_{p}}})|^{2} e^{-2 \varphi_{p}(\frac{u}{\sqrt{A_{p}}})} \,e^{-2 \varphi_{p}(\frac{\overline{v}}{\sqrt{A_{p}}})} }{A^{2n}_{p} \,e^{-2 \Sigma_{j=1}^{n}{\frac{\lambda^{p}_{j}}{A_{p}}|u_{j} - \overline{v_{j}}|^{2}}}}  \rightarrow 1
                                                                               \end{equation}
                                                                               locally uniformly on $\mathbb{C}^n _{u}\times \mathbb{C}^n _{v}$ as $p\rightarrow \infty$.
\end{thm}

\section{Asymptotic Normality}\label{lastsec}
\subsection{Random holomorphic sections and random zero currents of integration} Let $s\in H^{0}(X, L)\backslash \{0\}$. We denote by $Z_{s}$ the set of zeros of $s$. $Z_s$ is a purely $1$-codimensional analytic variety of $X$. By the symbol $[Z_{s}]$, we mean the current of integration on $Z_{s}$, defined by $$\langle [Z_{s}], \varphi \rangle =\int_{Z_{s}}{\varphi},$$here $\varphi\in \mathcal{D}^{n-1, n\emph{}-1}(X)$, which represents the space of test forms of type $(n-1, n-1)$. The random variable $\langle Z_{s}, \varphi \rangle$ is called the \textit{smooth linear statistic} for the zero set $Z_{s}$.

A complex random variable $W$ is said to be standard Gaussian in case $W= X + \sqrt{-1} Y$, where $X$ and $Y$ are i.i.d. centered real Gaussian distributions of variance $1/2$.
The Poincar\'{e}-Lelong formula is one indispensable instrument that will be of use in the proof of our main theorem. Let us state it for a zero current associated with the zero set of a holomorphic section $s_{p}\in H^{0}(X, L_{p})$. \begin{equation}\label{poinle}
  [Z_{s_{p}}]= dd^c \log{|s_{p}|_{h_{p}}} + c_{1}(L_{p}, h_{p}).
\end{equation}Here and throughout $dd^c= \frac{\sqrt{-1}}{\pi} \partial \overline{\partial}$.

\subsection{Asymptotic normality of random zero currents}
Now we are ready to prove our main theorem. For the proof, we use the arguments from \cite{SZ10} and \cite{Bay17}.

\begin{proof}[Proof of Theorem \ref{maint}]
To begin with, we modify the information about analytic functions from the introduction to suit our present setting. The normalized Gaussian processes $\alpha_{p}$ on $X$ will be constructed as follows: We take a measurable section $e_{L_{p}}$ of $L_{p}$ such that $e_{L_{p}}: X \rightarrow L_{p}$ with $|e_{L_{p}}(x)|_{h_{p}}=1$ for any $x\in X$. We pick now an orthonormal basis $\{s^{j}_{p}\}^{d_{p}}_{j=1}$ of $H^{0}(X, L_{p})$, where  $s^{j}_{p}= \varphi^{j}_{p} e_{L_{p}}$. Let us write \begin{equation}\label{pro}
                                                                f^{j}_{p}(x)=\frac{\varphi^{j}_{p}(x)}{\sqrt{K_{p}(x)}},\,\,j=1, 2, \ldots, d_{p}.
                                                              \end{equation} Notice that $|\varphi^{j}_{p}|=|s^{j}_{p}|_{h_{p}}$ and $\sum_{j=1}^{d_{p}}{|f^{j}_{p}(x)|^{2}}=1$ by the relation (\ref{diags}). Therefore, we can express a normalized complex Gaussian process on $X$ for each $p\in \mathbb{N}$ as follows: \begin{equation}\label{cgp}
                                                                         \alpha_{p}=\sum_{j=1}^{d_{p}}{b_{j}f^{j}_{p}},
                                                                       \end{equation}where the coefficients $b_{j}$ are i.i.d. complex Gaussian centered random variables with variance one. We observe that a random holomorphic section $s_{p}=\sum_{j=1}^{d_{p}}{b_{j}s_{p}^{j}}$ can be represented as \begin{equation}\label{sect}
                                                                        s_{p}=\sum_{j=1}^{d_{p}}{b_{j}s_{p}^{j}}=\sqrt{K_{p}(x)} \alpha_{p} e_{L_{p}},
                                                                     \end{equation} which indicates the presence of the normalized complex Gaussian process.  The relation (\ref{sect}) gives us that \begin{equation}\label{norm}
              |\alpha_{p}(x)|=\frac{|s_{p}(x)|_{h_{p}}}{\sqrt{K_{p}(x)}}.
            \end{equation}

We proceed to compute the covariance functions $\mathcal{C}_{p}$ of the complex Gaussian processes $\alpha_{p}$. We observe from the fact that the complex Gaussian random coefficients $b_{j}$ in (\ref{cgp}) are centered, i.i.d., and have variance one \begin{equation}\label{last}\text{Var}[b_{j}]=\mathbb{E}[|b_{j}|^{2}]=1,\,\,\,\mathbb{E}[b_{k}\bar{b_{l}}]=0\,\,\,\text{if}\,\,k\neq l.\end{equation} By the relation (\ref{cov}), (\ref{cgp}) and (\ref{last}), we have \begin{equation}\label{cov1}
                      \mathcal{C}_{p}(x, y)=\mathbb{E}[\sum_{j=1}^{d_{p}}{b_{j}f^{j}_{p}(x)}\,  \overline{\sum_{j=1}^{d_{p}}{b_{j}f^{j}_{p}(y)}}]=\sum_{j=1}^{d_{p}}{f^{j}_{p}(x) \overline{f^j_{p}(y)}}.
                    \end{equation}

Next, after a series of computations, we arrive at \begin{equation}\label{normb}
                                                    |K_{p}(x, y)|_{h_{p,x} \otimes h^{*}_{p, y}}= \sqrt{\sum_{j=1}^{d_{p}}{|\varphi^{j}_{p}(x)|^{2} |\varphi^{j}_{p}(y)|^{2}} + \sum_{\substack{j=1\\
                                                    j < k}}^{d_{p}}{2 \Re{(\varphi^{j}_{p}(x) \,\overline{\varphi^{j}_{p}(y)}\,\overline{\varphi^{k}_{p}(x)} \, \varphi^{k}_{p}(y)}})}
                                                  \end{equation}

By putting together (\ref{cov1}), (\ref{normb}) and (\ref{pro}), we find \begin{equation}\label{coveb}\widehat{\mathcal{K}}_{p}(x, y)= |\mathcal{C}_{p}(x, y)|.\end{equation}

Take $\lambda(\rho)=\log{\rho}$ and $(G, \sigma)=(X, \frac{\omega^{n}}{n!})$. In the rest of the proof, to make the notation lighter, we will write $d\vartheta_{X}:=\frac{w^{n}}{n!}$ for the (Riemannian) volume form on $X$. Let us fix a $(n-1, n-1)$ real-valued form $\phi$ with $\mathscr{C}^{3}$ coefficients. Then \begin{equation*}\label{testf}
                                                                                       dd^c \phi=\frac{\sqrt{-1}}{\pi}\partial \bar{\partial} \phi=\psi \, d\vartheta_{X},
                                                                                     \end{equation*}where the function $\psi$ is a real-valued $\mathscr{C}^1$ function on $X$. By invoking the Poincar\'{e}-Lelong formula (\ref{poinle}), the non-linear random functional given by (\ref{nonlin}) assumes the subsequent form in our case,
\begin{equation}\label{samev}\mathcal{F}^{\psi}_{p}(\alpha_{p})=\int_{X}{\big( \log{|s_{p}|_{h_{p}}- \log{\sqrt{K_{p}(x)}}} \big)\frac{\sqrt{-1}}{\pi}\partial \overline{\partial}\phi(x)}= \langle [Z_{s_{p}}], \phi \rangle + \zeta_{p, L_{p}},\end{equation}where $\zeta_{p, L_{p}}:= \big\langle-c_{1}(L_{p}, h_{p}) - dd^{c}\log{\sqrt{K_{p}(x)}}, \phi \big\rangle$, namely, $\zeta_{p, L_{p}}$ is a constant depending only on the geometric data $(L_{p}, h_p)$. Thanks to a standard property of variance, the expression (\ref{samev}) shows that $\mathcal{F}_{p}^{\psi}(\alpha_{p})$ and $\langle [Z_{s_{p}}], \phi \rangle$ have the equal variances.

For the remaining part of the proof, our goal is to validate the fulfillment of the requirements (i) and (ii) of Theorem \ref{sots} for the current setting. First, with $\lambda(\rho)=\log{\rho}$ being increasing, we only consider the case where $\nu=1$. To use both far off-diagonal and near-diagonal asymptotics, we split the integration regions accordingly: $d(x, y) \leq {\frac{\log{A_{p}}}{\sqrt{A_{p}}}}$ and $d(x, y) \geq {\frac{\log{A_{p}}}{\sqrt{A_{p}}}}$. Let us start with the simpler condition (ii). For the integral on the far off-diagonal set where $d(x, y) \geq {\frac{\log{A_{p}}}{\sqrt{A_{p}}}}$, by Theorem \ref{udeb}, we have,  \begin{equation*}
  \lim_{p \rightarrow \infty} \sup_{x\in X}\int_{d(x, y) \geq {\frac{\log{A_{p}}}{\sqrt{A_{p}}}}}{\widehat{\mathcal{K}}_{p}(x, y)d\vartheta_{X}(y)} \leq \lim_{p \rightarrow \infty} \sup_{x\in X}\int_{d(x, y) \geq {\frac{\log{A_{p}}}{\sqrt{A_{p}}}}}{G e^{-B \sqrt{A_{p}}d(x, y)}d\vartheta_{X}(y)}=0.
\end{equation*}For the integral over the near-diagonal set where $d(x, y) \leq {\frac{\log{A_{p}}}{\sqrt{A_{p}}}}$, due to the relations (\ref{coveb}) and (\ref{cov}), we get \begin{equation*}
                        \lim_{p \rightarrow \infty} \sup_{x\in X}\int_{d(x, y) \leq {\frac{\log{A_{p}}}{\sqrt{A_{p}}}}}{\widehat{\mathcal{K}}_{p}(x, y)d\vartheta_{X}(y)} \leq \lim_{p \rightarrow \infty} \sup_{x\in X}\int_{d(x, y) \leq {\frac{\log{A_{p}}}{\sqrt{A_{p}}}}}{1\,\,d\vartheta_{X}(y)}=0.
                        \end{equation*}

Our next step is to affirm the condition (i). For the integral on the far off-diagonal set where $d(x, y) \geq \frac{\log{A_{p}}}{\sqrt{A_{p}}}$, as $p \rightarrow \infty$, the integrand of the numerator approaches zero more rapidly compared to that of the denominator because, by Theorem \ref{udeb}, the corresponding decaying orders (to zero) for numerator and denominator are  $O(A^{-\epsilon}_{p})$\,and\, $O(A^{-\epsilon/2}_{p})$, respectively.

Finally, by the linearization (\ref{lincha}) on the neighborhood $P^{n}(x, R)$, where we have the K\"{a}hler coordinates at the point $x$ provided by Lemma \ref{refe}, and by fixing $x\in X$ in the support of $\phi$ at which the supremum is realized in the denominator, we verify that the lower limit below will be strictly positive on the near-diagonal set where $\{ \frac{|v|}{\sqrt{A_{p}}} \leq \frac{\log{A_{p}}}{\sqrt{A_{p}}}\}:$

\begin{equation*}\label{st22}
\liminf_{p \rightarrow \infty}{\frac{\int_{X}\int_{|v|\leq \log{A_{p}}}{\widehat{\mathcal{K}}_{p}^{2}(x, x+ \frac{v}{\sqrt{A_{p}}}) \psi(x) \psi(x+\frac{v}{\sqrt{A_{p}}})\,dv\, d\vartheta_{X}(x)}}{\int_{|v|\leq \log{A_{p}}}{ \widehat{\mathcal{K}}_{p}(x, x+ \frac{v}{\sqrt{A_{p}}})\,dv\,}}}>0.
\end{equation*}

Let \begin{equation*}\label{lastj}
      J(p):= \frac{\int_{X}\int_{|v|\leq \log{A_{p}}}{\widehat{\mathcal{K}}_{p}^{2}(x, x+ \frac{v}{\sqrt{A_{p}}}) \psi(x) \psi(x+\frac{v}{\sqrt{A_{p}}})\,dv\, d\vartheta_{X}(x)}}{\int_{|v|\leq \log{A_{p}}}{ \widehat{\mathcal{K}}_{p}(x, x+ \frac{v}{\sqrt{A_{p}}})\,dv\,}}.
    \end{equation*}Let us examine the numerator and the denominator separately. By using the left part of the inequality (\ref{feb}) for the denominator and the right part of the same inequality for the numerator, we get

\begin{align}\label{febsed}
 &\int_{X}\int_{|v|\leq \log{A_{p}}}{\widehat{\mathcal{K}}_{p}^{2}(x, x+ \frac{v}{\sqrt{A_{p}}}) \psi(x) \psi(x+\frac{v}{\sqrt{A_{p}}})\,dv\, d\vartheta_{X}(x)} \notag \\
& \quad \geq \int_{X} \int_{|v| \leq \log{A_{p}}} {(\frac{4}{5})^{2n} \frac{|K_{p}(x, x+\frac{v}{\sqrt{A_{p}}})|^{2}_{h_{p}}\, \psi(x)\, \psi(x+\frac{v}{\sqrt{A_{p}}})}{A^{2n}_{p} \,(1+D' \eta^{2/3}_{p})^{2}}\, dv \,d\vartheta_{X}(x)} :=I_{1}(p).
\end{align}

and

\begin{equation}\label{febused2}
\int_{|v|\leq \log{A_{p}}}{ \widehat{\mathcal{K}}_{p}(x, x+ \frac{v}{\sqrt{A_{p}}})\,dv\,}\\
 \leq \int_{|v| \leq \log{A_{p}}} {(\frac{4}{3})^{n} \frac{|K_{p}(x, x+\frac{v}{\sqrt{A_{p}}})|_{h_{p}}}{A^{n}_{p} \,(1 - D' \eta^{2/3}_{p})}\, dv} :=I_{2}(p),
\end{equation}which implies, by extracting the weight functions and multiplying and dividing both the integrand of $I_{1}$ and $I_{2}$ by $e^{-2\Sigma_{j=1}^{n}{\frac{\lambda^{p}_{j}}{A_{p}} |v_{j}|^{2}}}$ and $e^{-\Sigma_{j=1}^{n}{\frac{\lambda^{p}_{j}}{A_{p}} |v_{j}|^{2}}}$ respectively

\begin{align}\label{febusedd}
I_{1}(p)=\int_{X} \int_{|v| \leq \log{A_{p}}} {(\frac{4}{5})^{2n} \frac{|K_{p}(x, x+\frac{v}{\sqrt{A_{p}}})|^{2} \,e^{-2\varphi_{p}(x)} e^{-2\varphi_{p}(x+\frac{v}{\sqrt{A_{p}}})}\, \psi(x)\, \psi(x+\frac{v}{\sqrt{A_{p}}})}{A^{2n}_{p} \,(1+D' \eta^{2/3}_{p})^{2} e^{-2\Sigma_{j=1}^{n}{\frac{\lambda^{p}_{j}}{A_{p}} |v_{j}|^{2}}}}\,e^{-2\Sigma_{j=1}^{n}{\frac{\lambda^{p}_{j}}{A_{p}} |v_{j}|^{2}}} dv \,d\vartheta_{X}(x)}
\end{align}
and
\begin{equation}\label{febused22}
I_{2}(p)= \int_{|v| \leq \log{A_{p}}} {(\frac{4}{3})^{n} \frac{|K_{p}(x, x+\frac{v}{\sqrt{A_{p}}})| e^{-\varphi_{p}(x)} e^{-\varphi_{p}(x+ \frac{v}{\sqrt{A_{p}}})}}{A^{n}_{p} \,(1 - D' \eta^{2/3}_{p}) e^{-\Sigma_{j=1}^{n}{\frac{\lambda^{p}_{j}}{A_{p}} |v_{j}|^{2}}}}\,e^{-\Sigma_{j=1}^{n}{\frac{\lambda^{p}_{j}}{A_{p}} |v_{j}|^{2}}}\, dv}. \end{equation}

By (\ref{febsed}), (\ref{febused2}), (\ref{febusedd}) and (\ref{febused22}), we obtain

\begin{align}\label{sonn1}
   J(p) \geq \frac{I_{1}(p)}{I_{2}(p)} = \frac{\int_{X} \int_{|v| \leq \log{A_{p}}} {(\frac{4}{5})^{2n} \frac{|K_{p}(x, x+\frac{v}{\sqrt{A_{p}}})|^{2} \,e^{-2\varphi_{p}(x)} e^{-2\varphi_{p}(x+\frac{v}{\sqrt{A_{p}}})}\, \psi(x)\, \psi(x+\frac{v}{\sqrt{A_{p}}})}{A^{2n}_{p} \,(1+D' \eta^{2/3}_{p})^{2} e^{-2\Sigma_{j=1}^{n}{\frac{\lambda^{p}_{j}}{A_{p}} |v_{j}|^{2}}}}\,e^{-2\Sigma_{j=1}^{n}{\frac{\lambda^{p}_{j}}{A_{p}} |v_{j}|^{2}}} dv \,d\vartheta_{X}(x)}}{\int_{|v| \leq \log{A_{p}}} {(\frac{4}{3})^{n} \frac{|K_{p}(x, x+\frac{v}{\sqrt{A_{p}}})| e^{-\varphi_{p}(x)} e^{-\varphi_{p}(x+ \frac{v}{\sqrt{A_{p}}})}}{A^{n}_{p} \,(1 - D' \eta^{2/3}_{p}) e^{-\Sigma_{j=1}^{n}{\frac{\lambda^{p}_{j}}{A_{p}} |v_{j}|^{2}}}}\,e^{-\Sigma_{j=1}^{n}{\frac{\lambda^{p}_{j}}{A_{p}} |v_{j}|^{2}}}\, dv}}.
\end{align}

 Since $\psi \in \mathscr{C}^{1}$, it follows that $\psi(x+ \frac{v}{\sqrt{A_{p}}})=\psi(x)+ O(\frac{|v|}{\sqrt{A_{p}}})$. Now, as $p\rightarrow \infty$ on the inequality (\ref{sonn1}), making use of (\ref{ofdb}) and (\ref{limitl}) yields that $$\liminf_{p \rightarrow \infty}{J(p)} \geq {\frac{\int_{X}{\psi^{2}(x) d\vartheta_{X}(x)} \int_{\mathbb{C}^{n}}{(\frac{4}{5})^{2n}{\ e^{-\pi\Sigma_{j=1}^{n}{ |v_{j}|^{2}}}}dv}}{\int_{\mathbb{C}^{n}}{(\frac{4}{3})^{n} e^{-\frac{\pi}{2} \Sigma_{j=1}^{n}{|v_{j}|^{2}}}\, dv}}} =\frac{6^{n}}{5^{2n}} \int_{X}{\psi^{2}(x) d\vartheta_{X}(x)} >0,$$ which ends the proof. \end{proof}

In the papers \cite{CMM} and \cite{BCM}, the authors use a more general hypothesis between the first Chern currents and the K\"{a}hler form. More precisely, they assume  \begin{equation}\label{divergent}
	 c_{1}(L_{p}, h_{p}) \geq a_{p}\, \omega  \, \, \, \text{for all} \, \,  p\geq 1 \, \,  \text{and} \,  \,  a_{p} >0,  \, \, \text{such that} \, \, \, a_{p} \rightarrow \infty.
\end{equation}  In our setting, the role of $a_{p}$ is played by $A_{p}$ (which was defined as $A_{p}=\int_{X}{c_{1}(L_{p}, h_{p}) \wedge \omega^{n-1}}$ in \cite{CMM} and \cite{BCM}), and there are two different limits because of the diophantine approximation relation (\ref{appr}) between $\lambda^{p}_{j}$ and $A_{p}$. In the case of (\ref{divergent}), we do not have (\ref{limitl}) (and consequently (\ref{limitl1})). However, as Theorem 1.3 in \cite{CMM} shows, there still exists a limit: $\lim_{p\rightarrow \infty}\frac{K_{p}(x)}{\lambda^{p}_{1} \ldots \lambda^{p}_{n}}=(\frac{2}{\pi})^{n}$. Despite the existence of the limit in terms of $\lambda^{p}_{j}$, we do not know whether the limit $\lim_{p\rightarrow \infty}\frac{\lambda^{p}_{j}}{A_{p}}$ exists, which is crucial in the proof of Theorem \ref{maint}. Therefore, the arguments followed in this paper cannot be used to prove a central limit theorem in this framework.

\vspace{5mm}

\textbf{Acknowledgement:} We thank Turgay Bayraktar for his interest in and comments on this work. We thank the anonymous referee for his/her careful review, suggestions and corrections, which helped enhance the presentation of this paper.


{}

 \end{document}